\documentclass[12pt]{amsart}

\usepackage{latexsym}
\usepackage{amscd, amsfonts, mathrsfs, amsmath, amssymb, amsthm, mathtools}
\usepackage{baskervald}

\usepackage[color,arrow,matrix,curve]{xy}

\usepackage[colorlinks=true,
linkcolor=black,
citecolor=black,
filecolor=black,
urlcolor=black,
bookmarks=true,
bookmarksopen=true,
bookmarksopenlevel=3,
plainpages=false,
pdfpagelabels=true]{hyperref}

\usepackage{cleveref}
\usepackage{wasysym}

\newtheorem{thm}{Theorem}[section] 
\crefname{thm}{theorem}{theorems}

\crefname{algo}{algorithm}{algorithms}

\crefname{claim}{claim}{claims}

\crefname{conj}{conjecture}{conjectures}
\newtheorem{cor}[thm]{Corollary}
\crefname{cor}{corollary}{corollaries}
\newtheorem{defn}[thm]{Definition}
\crefname{defn}{definition}{definitions}

\crefname{exer}{exercise}{exercises}

\crefname{Ex}{example}{examples}
\newtheorem{lemma}[thm]{Lemma}
\crefname{lemma}{lemma}{lemmas}
\newtheorem{prop}[thm]{Proposition}
\crefname{prop}{proposition}{propositions}

\crefname{prob}{problem}{problems}

\crefname{rem}{remark}{remarks}

\crefname{notation}{notation}{notation}

\crefname{ques}{question}{questions}

\def\ra{\rightarrow}

\def\Stab[#1]{\text{Stab}_G \left( #1 \right)}
\def\ZZ{\mathbb{Z}}

\def\ed[#1][#2]{\text{ed}_{#1} \left( #2 \right)}
\def\cd[#1]{\text{cd}_k \left( #1 \right)}
\def\cd[#1][#2]{\text{cd}_{#1} \left( #2\right)}
\def\edp[#1]{\text{ed}_k\left( #1 ; p \right)}

\def\O{\mathcal{O}}

\def\c{\mathcal{c}}

\def\ZZ{\mathbb{Z}}

\def\HGm[#1]{H^2\left( #1_{ét}, \Gm \right)}

\def\Brone[#1][#2]{\text{Br}_1 \left( #1\right)_{#2}}
\def\Brone[#1]{\text{Br}_1 \left( #1\right)}

\def\tor[#1][#2]{\prescript{}{#1}{#2}}

\def\Gm{\mathbb{G}_m}
\def\etale{\'etale\ }
\def\Poincare{Poincar\'e\ }

\newcommand{\kbar}{\overline{k}}

\newcommand{\Ebar}{\overline{E}}
\newcommand{\Hom}{\text{Hom}}

\def\xcoor[#1]{\mathbf{x}\left( #1 \right)}
\def\ycoor[#1]{\mathbf{y}\left( #1 \right)}
\renewcommand{\phi}{\varphi}

\DeclareMathOperator{\Gal}{Gal}
\DeclareMathOperator{\Br}{Br}
\DeclareMathOperator{\Spec}{Spec}
\DeclareMathOperator{\Pic}{Pic}
\DeclareMathOperator{\End}{End}
\DeclareMathOperator{\Div}{Div}
\DeclareMathOperator{\Ext}{Ext}
\DeclareMathOperator{\GL}{GL}

\DeclareMathOperator{\PGL}{PGL}
\DeclareMathOperator{\stab}{Stab}

\DeclareMathOperator{\Prin}{Prin}
\DeclareMathOperator{\id}{id}
\DeclareMathOperator{\Proj}{Proj}
\DeclareMathOperator{\Jac}{Jac}

\title{A Moduli Interpretation of Untwisted Binary Cubic Forms} 
\author{Rajesh S. Kulkarni and Charlotte Ure} 
\date{}

\begin{document}

\begin{abstract}
	We give a moduli interpretation to the quotient of (nondegenerate) binary cubic forms with respect to the natural $\GL_2$-action on the variables. In particular, we show that these $\GL_2$ orbits are in bijection with pairs of $j$-invariant $0$ elliptic curves together with $3$-torsion Brauer classes that are invariant under complex multiplication. The binary cubic generic Clifford algebra plays a key role in the construction of this correspondence. 
\end{abstract}

\maketitle

\section*{Introduction}
\label{sect:introduction}

The subject of homogeneous forms has fascinated mathematicians for centuries as it relates various branches such as algebraic geometry and number theory. Our particular interest is in binary forms which have been studied in great detail.  For example, the invariants and covariants of binary forms of degree $3$ and $4$ were classically described by Schur et al. \cite{schur}. Eisenstein also investigated binary cubic forms (over ${\mathbb Z}$) by constructing a finite map to the space of quadratic forms. He counted the number of elements in a fiber as a step towards counting the number of representations of a positive integer as a given cubic form  \cite{Eisenstein}. Binary forms of higher degree were studied, for example, by Franklin and Sylvester \cite{Sylvester}.

Binary cubic forms over ${\mathbb Z}$ have been of great interest in number theory in relation to counting cubic extensions of ${\mathbb Q}$. The idea is as follows: A cubic field extension of ${\mathbb Q}$ is generated by attaching a root of a polynomial $f(x) \in {\mathbb Q}[x]$. However, different polynomials may generate the same field. For example, the leading coefficient plays no role in generating the cubic field and the polynomial obtained from a generating polynomial via a linear transformation of the variable also generates the same field. In this vein, a result of Davenport \cite{Davenport1,Davenport2} gives an asymptotic expression for the number of cubic fields as the discriminant approaches infinity.

In recent times, binary cubic forms over ${\mathbb Z}$ were studied by various authors: Bhargava-Gross \cite{BhargavaGross}, Gan-Gross-Savin \cite{Gan-Gross-Savin:Fourier-Coefficients-of-modular-forms}, and Wood \cite{Wood:Rings-and-ideals}. The main goal of these works is to use the {\it moduli space} of binary cubic forms to count the number of cubic field extensions of ${\mathbb Q}$. These works begin with the parameter space of all nondegenerate binary cubic forms, and then consider the quotient by $\GL_2$. The action of $\GL_2$ is twisted by the determinant to account for the fact that different polynomials may generate the same cubic field.

In this article, we investigate the natural question: Can we describe the moduli space obtained as the quotient of the space $U$ of (nondegenerate) binary cubic forms under the natural {\it untwisted} action of $\GL_2$? This raises other related questions: Does the family of genus 1 curves over $U$ descend to this quotient? What about the associated family of Jacobians of these curves? We study these questions over any base field of characteristic different from $2$ and $3$ containing a primitive third root of unity $\omega$. 

It turns out that there is a unifying theme that ties up these aspects. There is a noncommutative algebra, called the {\it generic Clifford algebra} that acts as the universal algebra where the forms are linearized. The generic Clifford algebra $A_{n, d}$ of degree $d$ in $n$ variables was introduced by Chan, Young, and Zhang in \cite{Chan-Young-Zhang:Discriminant-formulas-and-applications}. It is the quotient of the tensor algebra $TV$ given by 
\[
A_{n,d} = \frac{TV}{\left< \left[ u^d, v \right] : u,v \in V \right> },
\]

\noindent where $V$ is a $k$-vector space of dimension $n$ and $\left[ u^d, v \right]$ denotes the commutator. The algebraic (ring-theoretic) properties of this algebra have been studied in some cases. For example, when $d=2$, this algebra is a connected, graded, Artin-Schelter regular, Auslander regular, Cohen Macauly noetherian domain \cite{Chan-Young-Zhang:Discriminant-formulas-and-applications}. 

Our interest is in the binary cubic case ($n = 2, d= 3$). The algebra $A = A_{2,3}$ is called the \textit{binary cubic generic Clifford algebra}. It is easy to see that the algebra $A$ can be described by generators and relations as

\[A = k \left< x,y : x^3y = yx^3, xy^3 = y^3 x, x^2 y^2 + (xy)^2 = y^2 x^2 + (yx)^2 \right>. \]

This algebra has been studied in the context of the classification of Artin-Schelter regular algebras. Wang and Wu showed that it is an Artin-Schelter regular, strongly noetherian, Auslander regular, Cohen Macauly algebra of global dimension $5$ \cite{Wang-Wu:AclassofASregularalgebrasofdimensionfive}. Its representation theoretic properties were explored by Wang and Wang in \cite{Wang-Wang:A-note-on-generic-Clifford-algebras-of-binary-cubic-forms}. We explore other algebraic properties of $A$ in \Cref{algebraic}. In particular, we show that the center of $A$ is isomorphic to the affine coordinate ring of a relative projective curve over $\mathbb{A}^4$ that is elliptic over a Zariski open dense subset $U$. This subset $U$ corresponds to the set of nondegenerate binary cubic forms over $k$. 

One of our main results is the following combination of \Cref{thm:ADelta-is-Azumaya} and \Cref{thm:Brauer-class-nontrivial} that relates the algebra $A$ to the moduli problem discussed above: 
\begin{thm}\label{thm1} 
	The binary cubic generic Clifford algebra $A$ defines a Brauer class over a relative elliptic curve over $U$. The associated Brauer class is not trivial for any base field $k$. 
\end{thm}

We will now describe how the binary cubic generic Clifford algebra relates to the $\GL_2$ action on binary cubic forms described above. The natural $\GL_2$-action on the two dimensional vector space $V$ extends to the algebra $A$. In particular, it induces an action on the center $Z(A)$. This action is compatible with the $\GL_2$-action on the affine space $\mathbb{A}^4$ and it preserves the open subset $U$ of nondegenerate binary cubic forms (see \Cref{sect:moduli}). Combining the binary cubic generic Clifford algebra with this $\GL_2$-action on nondegenerate binary cubic forms, we prove that the associated Brauer class descends to the quotient stack \Cref{prop:descend-alpha}. 

For our investigation of the moduli space of binary cubic forms, we use yet another algebra called {\it the generalized Clifford algebra} associated to a form $f$ \cite{Childs:Linearizing-of-n-ic-forms-and-generalized-clifford-algebras,Revoy:AlgebresdeCliffordetalgebresexterieures,Roby:AlgebresdeClifforddesformespolynomes}. Let $f$ be a binary homogeneous form of degree $3$ over a field $k$. The \textit{generalized Clifford algebra $A_f$ associated to $f$} is the quotient of the tensor algebra 
\[
A_f = \frac{TV}{\left<v^{3} - f(v) :  v \in V\right>}.
\]

\noindent This is a natural generalization of the classical Clifford algebra associated to a quadratic form. If $f$ is given as $f(u,v) = au^3 + 3bu^2v + 3cuv^2 + dv^3$, we can write the algebra $A_f$ using generators and relations as 

$$A_f = k\left< x,y : x^3 = a , x^2 y+ xyx + yx^2 = 3b, xy^2 + yxy+ y^2 x = 3c, y^3 = d \right>.$$

\noindent This algebra is an Azumaya algebra over its center, the complement of zero in the Jacobian of $w^3 - f(u,v)$ \cite{Haile:On-the-Clifford-Algebra-of-a-binary-cubic-form,Heerema:an-algebra-determined-by-a-binary-cubic-form}. Furthermore, the associated Brauer class extends to the Jacobian and is trivial if and only if the curve $w^3-f(u,v)$ admits a $k$-rational point \cite{Haile:WhenistheCliffordAlgebraofabinarycubicformsplit}. These works also give an explicit equation for the Jacobian as  $s^2 = \gamma^3 + \tfrac14 \Delta$, where $\Delta$ is the discriminant of the form $f$.

We continue with the discussion of the results of this article. We show that the action of $\GL_2$ on the set of nondegenerate binary cubic forms extends to an action on the center $Z(A)$ of $A$ (see \Cref{sect:moduli}). In \Cref{thm:Moduli-space}, we give an answer to the moduli problem of nondegenerate binary cubic forms with respect to (untwisted) $\GL_2$-equivalence. 

\begin{thm}\label{thm2} 
	Let $k$ be a field that contains a primitive third root of unity. There is a one-to-one correspondence between
$$
\left\{
\begin{gathered} 
\text{(untwisted) }
	\GL_2-\text{orbits}\\
	\text{of binary cubic forms over $k$}
\end{gathered}
\right\}
\overset{1:1}\longleftrightarrow 
\left\{
\begin{gathered} 
\text{$k$-isomorphism classes of pairs $(E,\alpha)$}\\
\text{of elliptic curves $E$ with $j$-invariant $0$ and}\\
\text{$3$-torsion Brauer classes $\alpha \in \Br(E)/\Br_0(E)$}\\
\text{ invariant under complex multiplication of $E$}
\end{gathered} 
\right\},
$$
where $\Br_0(E)$ denotes the constant Brauer classes that are in the image of the natural map $\Br(k) \rightarrow \Br(E)$. 
\end{thm} 
This correspondence can be described as follows: For a binary cubic form $f$, let $(E_f,\alpha_f)$ be the pair given by the Jacobian $E_f$ of the curve $w^3 - f(u,v)$, and the Brauer class $\alpha_f \in \Br(E_f)$ of the Clifford algebra associated to $f$. For the reverse direction, let $(E,\alpha)$ be a pair as above. Consider the short exact sequence 

$$\xymatrix{ 0 \ar[r]& \Br(k) \ar[r] & \Br(E) \ar[r] & H^1(k,E) \ar[r] & 0}$$ 

\noindent induced by the Hochschild-Serre spectral sequence. We compute the image of the Brauer class $\alpha$ in $H^1(k,E)$ (see \Cref{sect:moduli}). This image is a cyclic twist $X$ of $E$. Such twists (or rather their relative Brauer groups) were studied by Ciperiani and Krashen in \cite{Ciperiani-Krashen:Relative-Brauer-Groups-of-geneus-1-curves} and computed explicitly by Haile, Han and Wadsworth in \cite{Haile-Wadsworth-Han:Cyclic-Twists}. In particular, such a cyclic twist $X$ is of the form $w^3 - f(u,v)$ for some nondegenerate binary cubic form $f$. \\

There are many natural questions that can be pursued following this work. One natural avenue is to explore noncommutative algebraic interpretation of extension of results of this article to pairs of (all) ellipic curves and Brauer classes. As a first step in this program, the most natural path is to generalize this modular interpretation to pairs of $j$-invariant $1728$ and Brauer classes invariant under complex multiplication. More generally, the questions answered in this article can be raised for Jacobians of higher genus curves as well. In the case of CM Jacobians, 
there is a noncommutative analog the generic Clifford algebra of higher degree forms. There are multiple obstacles that only appear when $d=3$ and $n>2$ or $d>3$. 
To overcome this difficulty, Kulkarni studied a quotient of $A_f$, the reduced Clifford algebra, that is Azumaya over its center \cite{Kulkarni:the-extension-of-the-reduced-clifford-algebra-and-its-brauer-class, Kulkarni:On-the-Clifford-algebra-of-a-binary-form}. We anticipate that this quotient construction can be made universal and used in answering similar questions about CM Jacobians and CM invariant Brauer classes on them. \\ 

This paper is organized as follows: In \Cref{algebraic}, we explore algebraic properties of the binary cubic generic Clifford algebra $A$. In particular, we compute the center of $A$ explicitly and prove that $A$ is a free module of rank $18$ over a subring of its center. \Cref{sect:geometric} is devoted to a geometric construction of the Brauer class of the Clifford algebra that enables us to prove \Cref{thm1}. We conclude the section by showing that even though the entire algebra $A$ is not Azumaya over its center, it is a maximal order. In \Cref{sect:moduli}, we prove \Cref{thm2} and explain the details involved with the $\GL_2$-action on (nondegenerate) binary cubic forms. We connect the moduli space problem to the generic cubic Clifford algebra by proving that the Brauer class descends to the quotient stack (\Cref{prop:descend-alpha}). In \Cref{appendix} we give some additional background on coverings of elliptic curves that are essential to the moduli space correspondence. \\

\textbf{Acknowledgements. } The authors thank James Zhang for his original questions on the generic Clifford algebra which led to this article. The authors would also like to thank Quanshui Wu and Xingting Wang for pointing out the connection to the classification of AS-regular algebras of global dimension five. The second author also thanks Anthony V\'arilly-Alvarado for a useful conversation related to this project. The first author (RK) was partially supported by the NSF Grant DMS-1305377 at the beginning of this project.

\section{Ring Theoretic Properties of the Binary Cubic Generic Clifford Algebra}\label{algebraic}

In this computational section, we study algebraic and ring theoretic properties of the binary cubic generic Clifford algebra ${A}$. In particular, we show that results of Haile \cite{Haile:On-the-Clifford-Algebra-of-a-binary-cubic-form, Haile:WhenistheCliffordAlgebraofabinarycubicformsplit} and Hereema \cite{Heerema:an-algebra-determined-by-a-binary-cubic-form} concerning the Clifford algebra associated to a binary cubic form extend to the generic case. We compute the center $Z$ of $A$ in \Cref{thm:center-of-A} explicitly and prove that it is the affine coordinate ring of a relative curve over $\mathbb{A}^4$, that is elliptic over a Zariski-open subset. As part of this proof, we also show that $A$ is a free module of rank $18$ over a subring of its center and compute a basis in \Cref{thm:basisofA}. \\

Fix two positive integers $n$ and $d$ and assume that $d$ is prime to the characteristic of the field $k$. Let $V$ be a vector space of dimension $n$. 

\begin{defn} 
	The \emph{(generalized) Clifford algebra $A_f$ associated to the homogeneous form $f$ on $V$} is the quotient of the tensor algebra $TV$ by the ideal generated by elements of the form 
	$v^d - f(v)$ for all $v \in V$. 
\end{defn}

We will focus on the binary cubic case ($n=2$ and $d=3$). In this paper, we study the generic Clifford algebra. This algebra was first introduced by Chan, Young, and Zhang in \cite{Chan-Young-Zhang:Discriminant-formulas-and-applications}.

\begin{defn}
	The \emph{generic Clifford algebra of degree $d$ in $n$ variables $A_{n,d}$} is the quotient of the tensor algebra $TV$ by the ideal generated by elements of the form $[u,v^d]$ with $u,v \in V$. Here $[u, v^d] = uv^d - v^d u$.
\end{defn}

For any homogeneous form $f$ of degree $d$, the Clifford algebra associated to $f$ is a quotient of the generic Clifford algebra $A_{n,d}$. For any element $a$ in the generic Clifford algebra, we denote by $a_f$ the image in the Clifford algebra associated to $f$ under the quotient map $A \rightarrow A_f$. \\

Let $k$ be a field of characteristic different from $2$ and $3$ containing a primitive third root of unity $\omega$. We are interested in the binary cubic generic Clifford algebra 
$$A = A_{2,3} = k\left<x,y: x^3 y = yx^3,  xy^3 = y^3 x, x^2 y^2 + (xy)^2 = y^2 x^2 + (yx)^2 \right>.$$ 
Wang and Wu proved in \cite{Wang-Wu:AclassofASregularalgebrasofdimensionfive} that $A$ is Artin-Schelter regular, strongly noetherian, Auslander regular, and Cohen-Macauly. \\

We fix some special elements of the algebra, that will be used throughout this section:
\begin{itemize} 
	\item $\alpha = x^2 y + xyx + yx^2$,
	\item $\beta = xy^2 + yxy + y^2x$, 
	\item $\gamma = (xy)^2 - y^2 x^2 = (yx)^2 - x^2y^2$, and  
	\item $\delta = yx - \omega xy$. 
\end{itemize} 
Furthermore, let $\mathcal{S} = k[x^3,y^3,\alpha,\beta,\gamma]$. The main result of this section is the following theorem.  

\begin{thm}\label{thm:center-of-A}The center of $A$ is 
	$\mathcal{S}[\delta^3]= k \left[ x^3, y^3, \alpha,\beta, \gamma ,\delta^3\right] $. Additionally, the center is isomorphic to the affine coordinate ring of a relative quasi-projective curve over the affine four-space $\mathbb{A}^4$, that is elliptic over a Zariski-open subset $U$. This relative curve is described by the equation 
	\begin{align} \label{eqn:elliptic-curve-eqn}
	s^2 = \gamma^3 + \tfrac14 \Delta\left( x^3, \alpha, \beta,y^3\right),\end{align} 
	where $\Delta$ is the discriminant and $s= \delta^3 - \tfrac{ 3\omega(1-\omega)x^3 y^3 + (1+2w^2) \alpha \beta }2$. 
\end{thm}

We identify $\mathbb{A}^4$ with the space of binary cubic forms by identifying a point $(a,b,c,d)$ in $\mathbb{A}^4(k)$ with the binary cubic form $f(u,v) =  au^3 + 3bu^2v + 3 cuv^2 + dv^3$. Under this correspondence $U$ is the affine subset of $\mathbb{A}^4$ given by nondegenerate forms. Thus \Cref{thm:center-of-A} implies that the center of $A$ defines a relative elliptic curve over the space of nondegenerate binary cubic forms. \\

The elements $x^3, y^3, \alpha,$ and $\beta$ are obviously elements of the center $Z$ of $A$. Furthermore, these elements are algebraically independent over $k$ and so $k\left[ x^3,y^3,\alpha,\beta \right] \subseteq Z$ is a polynomial ring in four variables. The following two lemmata show that 	$\mathcal{S}[\delta^3]= k \left[ x^3, y^3, \alpha,\beta, \gamma ,\delta^3\right] $ is a subring of the center $Z$. 

\begin{lemma}\label{lemma:c-in-center} The element 
	$\gamma= (xy)^2 -y^2x^2 = (yx)^2 -x^2y^2$ is in $Z$. Furthermore, $\gamma$ is not integral over $k\left[ x^3, y^3, \alpha, \beta \right]$. In particular, the subring $\mathcal{S} = k\left[ x^3, y^3, \alpha, \beta,\gamma\right]$ of $Z$ is isomorphic to the polynomial ring in five variables. 
\end{lemma}

\begin{proof}
	For the first part of the statement, it is sufficient to show that $\gamma$ commutes with $x$ and $y$, which can be seen by the following straightforward calculations:
	\begin{equation*} 
	\begin{gathered} 
	x \gamma = x (yx)^2 -  x x^2 y^2 = (xy)^2 x - y^2 x^2 x = \gamma x,\\
	y \gamma = y (xy)^2 - y y^2 x^2 =  (yx)^2y - x^2 y^2 y = \gamma y.
	\end{gathered} 
	\end{equation*} 
	For the second part of the statement, suppose that $\gamma$ is integral over $k\left[ x^3, y^3, \alpha, \beta\right].$ Then $k\left[ x^3, y^3, \alpha, \beta\right][\gamma]$ is a finitely generated module over $k\left[ x^3, y^3, \alpha, \beta\right]$ with generators $a_1, \ldots, a_n$. For any binary cubic form $f$, denote by $a_f$ the image of $a \in A$ under the quotient map $A \rightarrow A_f$. Then $k[\gamma_f]$ is generated by ${a_1}_f, \ldots {a_n}_f$, which is a contradiction to the fact that $\gamma_f$ is transcendental over $k$ for any nondegenerate diagonal form \cite[Lemma 1.9]{Haile:On-the-Clifford-Algebra-of-a-binary-cubic-form}.  
\end{proof}

\begin{lemma}\label{lem:delta3-in-center}
	Let $\delta = yx- \omega xy$, then $\delta^3$ is in the center of $A$.
\end{lemma}

\begin{proof}Direct computations show that $$\delta x = \omega^2 x\delta + \alpha \qquad \text{and} \qquad  y\delta = \omega^2 \delta y + \beta.$$ Therefore 
	$$\delta^3 x = \omega^2 \delta x \delta^2 + \alpha \delta^2 = \omega x^2 \delta x + \omega^2 \alpha \delta^2 + \alpha \delta^2 = x^3 \delta + \omega \alpha \delta^2 + \omega^2 \alpha \delta^2 + \alpha \delta^2 = x \delta^3.$$ Similarly also $\delta^3 y = y \delta^3$. 
\end{proof}

Haile proved in \cite[Lemma 1.6]{Haile:On-the-Clifford-Algebra-of-a-binary-cubic-form} that $A_f$ is free of rank 18 over the subring $k[\gamma_f]$ of its center, where is a diagonal nondegenerate binary cubic form. We recover this statement in the following theorem for the generic case. 

\begin{thm}\label{thm:basisofA}
	The algebra $A$ is a free module of rank 18 over the subring $\mathcal{S}$ of the center of $A$. A set of generators is given by the following list $b_0, \ldots, b_{17}$
	\begin{equation*}
	\begin{gathered}
	1, \\
	x, y,\\
	x^2, xy, yx,  y^2,\\
	x^2y, xy^2, y^2x, yx^2,\\
	x^2y^2, xyxy, xyx^2, y^2xy,\\
	x^2y^2x, xyxy^2,\\
	x^2y^2xy.
	\end{gathered}
	\end{equation*} 
\end{thm}
Note that this basis of $A$ over $\mathcal{S} = k[x^3,y^3,\alpha,\beta,\gamma]$ is mapped to the basis of $A_f$ over $k[\gamma_f]$, for $f$ any diagonal form, that Haile describes in \cite[Lemma 1.6]{Haile:On-the-Clifford-Algebra-of-a-binary-cubic-form}. Before we proceed to the proof of the theorem, we need to show that this image also determines a basis for non-diagonal binary cubic forms.  

\begin{lemma}\label{lemma:Cffree}Let $f$ be a nondegenerate binary cubic form (not necessarily diagonal) over $k$ with discriminant $\Delta$ and suppose that $\sqrt{-108 \Delta} \in k$. Then $A_f$ is free of rank $18$ over $k[\gamma_f]$ with basis $\left\{(b_i)_f : 0 \leq i \leq 17\right\},$
	where $b_i$ is defined in \Cref{thm:basisofA} and $\left( b_i \right)_f$ denote the images of $b_i$ under the homomorphism $A \rightarrow A_f$. 
\end{lemma}

\begin{proof}
	Let $f(u,v) = a u^3 + b u^2 v + c uv^2 + d v^3$ be any nondegenerate binary cubic form with discriminant $\Delta$ such that $\sqrt{D} \in k$, where $D = -108 \Delta$. In this case $f$ is diagonalizable over $k$ via the linear transformation 
	$$u \mapsto (\sqrt{D} + s ) \tilde{u} + (\sqrt{D} -s ) \tilde{v} \qquad v \mapsto -r \tilde{u} + r \tilde{v}, $$where $r= \frac13ac-\frac19b^2$, $2s= ad-\frac{1}{9}bc$, $t= \frac13bd-\frac19c^2$, and $D = s^2 - rt$. The transformation results in a diagonal form $g(\tilde{u}, \tilde{v})$. By \cite[Proposition 1]{Childs:Linearizing-of-n-ic-forms-and-generalized-clifford-algebras}, $C_f$ is canonically isomorphic to $C_g$ via an isomorphism $\Phi$. Further, $A_g$ is a free module of rank $18$ over $k[\gamma_g]$ by \cite[Lemma 1.12]{Haile:On-the-Clifford-Algebra-of-a-binary-cubic-form}. Since $A_f$ and $A_g$ are isomorphic, $A_f$ is also a free module of rank $18$ over $k[\Phi^{-1}(\gamma_g)]$. Direct calculations shows that $\Phi(\gamma_f) = 4 r^2 D \gamma_g$ and so $k\left[\Phi^{-1}(\gamma_g)\right] = k\left[\gamma_f\right]$. This implies that $A_f$ is a free module of rank $18$ over $k[\gamma_f]$. As $\{(b_i)_f : 0 \leq i \leq 17\}$ is a generating set for $A_f$ over $k[\gamma_f]$ it has to be linearly independent by dimension counting.
\end{proof}

We are now ready to proceed to the proof of the theorem. 
\begin{proof}[Proof of Theorem \ref{thm:basisofA}]
	We will first show that $\{b_i: 0\leq i \leq 17 \}$ generate $A$ over $\mathcal{S}$. Remark that it is sufficient to show that for any $b_i$, the monomials $b_ix$ and $b_iy$ are in the module generated by $\{b_i: 0\leq i \leq 17 \}$. We will only include the case $b_{17} x$ as the other calculations are similar:
	\begin{align*}
	b_{17} x = x^2y^2xyx &= x^2 y \gamma + x^2yx^2y^2 \\
	&= \gamma x^2y +\alpha x^2y^2 - xyx^3y^2 -yx^4y^2\\
	&= \gamma x^2y + \alpha x^2y^2 - x^3y^3 x - x^3 yxy^2\\
	&= \gamma x^2y + \alpha x^2y^2 - x^3y^3 x - x^3 \beta y + x^3 xy^3 + x^3 y^2 xy\\
	&= \gamma x^2y + \alpha x^2y^2 - x^3 \beta y + x^3 y^2 xy.
	\end{align*}
	It remains to prove that the $b_i$ are linearly independent. From now on assume that $k$ is algebraically closed and suppose that there was some relation $0 = \sum_{i=0}^{17} \lambda_i b_i$ for some $\lambda_i \in \mathcal{S}$. For any nondegenerate binary cubic form $f$, we deduce a relation $\sum_{i=0}^{17} (\lambda_i)_f (b_i)_f$, with $\lambda_i \in k[\gamma_f]$. Since $k$ is assumed to be algebraically closed, $\sqrt{\Delta} \in k$, where $\Delta $ is the discriminant of $f$. By Lemma \ref{lemma:Cffree} the set $\{ b_i: 0 \leq i \leq 17\}$ is linearly independent over $k\left[\gamma_f\right]$, and thus $(\lambda_i)_f =0$. Since $\lambda_i \in \mathcal{S} = k[x^3,y^3,\alpha,\beta][\gamma]$, it is possible to write each $\lambda_i$ as a polynomial in $\gamma$, i.e. $\lambda_i = \sum_{j} \mu_{i,j} \gamma^j$ with $\mu_{i,j} \in k[x^3,y^3,\alpha,\beta]$. Then $0 = \left( \lambda_i\right)_f =  \sum_{j} (\mu_{i,j})_f \gamma_f^j$ for any nondegenerate $f$. By \cite[Lemma 1.9]{Haile:On-the-Clifford-Algebra-of-a-binary-cubic-form} the element $\gamma_f$ is transcendental over $k$ and therefore $(\mu_{i,j})_f = 0$. Since the set $U$ of nondegenerate polynomials is dense in $\mathbb{A}^4$, the polynomials $\mu_{i,j}$ must vanish on all of $\mathbb{A}^4$. In particular $\mu_{i,j} = 0$ in $A$ and thus each $\lambda_i$ is trivial as well. This finishes the proof. 
\end{proof}

This explicit description of $A$ as a module over a subring of its center now enables us to calculate the center explicitly. 

\begin{proof}[Proof of Theorem \ref{thm:center-of-A}]
	Recall that by Lemma \ref{lemma:c-in-center} and Lemma \ref{lem:delta3-in-center} the ring $\mathcal{S}[\delta^3]$ is a subring of the center $Z$ of $A$. It remains to show that this is all of $Z$. We first compute $\delta^3$ directly in terms of the basis given in \Cref{thm:basisofA}.  
	\begin{align*}
	\delta^3 =& (yx)^3 - \omega \left( (yx)^2 (xy) + (yx)(xy)(yx) + (xy)(yx)^2\right) \\
	&+ \omega^2 \left( (yx)(xy)^2 + (xy)(yx)(xy) + (xy)^2(yx) \right) - (xy)^3 \\
	=& \gamma  b_ 5 - \gamma b_4\\
	&- \omega \left( - \alpha \beta +3 cb_4 - \beta b_7 + 2\alpha b_8+ \alpha b_9 + \beta b_{10} + 3 b_{17} \right) \\
	&+ \omega^2 \left( 2\alpha \beta - 3 x^3y^3 - 3cb_4 + \beta b_7 - 2 \alpha b_8 - \alpha b_9 - \beta b_{10} - 3 b_{17}  \right)\\
	=& \gamma b_5 - \gamma b_4 + (- \omega - \omega^2 )\left( 3cb_4 - \beta b_7 + 2\alpha b_8 + \alpha b_9 + \beta b_{10} + 3 b_{17}\right)\\
	&+ (\omega+ 2 \omega^2) \alpha \beta - 3 \omega^2 x^3y^3 \\ 
	=& \gamma b_5 +2 \gamma b_4 - \beta b_7 + 2\alpha b_8 + \alpha b_9 + \beta b_{10} + 3 b_{17}+ (\omega+ 2 \omega^2) \alpha \beta - 3 \omega^2 x^3y^3. 
	\end{align*}
	Let $\tilde{a}\in Z$. By \cref{thm:basisofA} there exist unique $\lambda_i \in \mathcal{S}$ so that $\tilde{a} = \sum_{i=0}^{17} \lambda_i b_i$. Since $\tilde{a}$ is central, it has to commute with $x$ and $y$. By comparing the coefficients of $\tilde{a}x$ and $x\tilde{a}$, and $\tilde{a}y$ and $y\tilde{a}$, respectively, we see that any element in the center has to be of the form 
	\begin{align*} 
	\tilde{a} &= \nu + 2 \gamma \lambda b_4 + \gamma \lambda b_5 - \beta \lambda b_7 + 2 \alpha \lambda b_8 + \alpha \lambda b_9 + \beta \lambda b_{10} + 3 \lambda b_{17}\\
	&= \tilde{\nu} + \lambda \delta^3\end{align*} 
	for some $\lambda,\nu, \tilde\nu \in \mathcal{S}$. This concludes the proof that $\tilde{a} \in \mathcal{S}[\delta^3]$.\\
	
	For the second part of the statement, we verify with a direct computation that $\delta^3$ satisfies the equation $$\delta^6 = 3\omega(1-\omega)x^3 y^3 \delta^3 + (1+2w^2) \alpha \beta \delta^3  + \gamma^3 - x^3 \beta^3 - y^3 \alpha^3 + \alpha^2 \beta^2$$ 
	Denote $$s= \delta^3 - \tfrac{ 3\omega(1-\omega)x^3 y^3 + (1+2w^2) \alpha \beta }2.$$
	Then the center of $A$ is $\mathcal{S}[\delta^3 ] = \mathcal{S}[s]$ and 
	\begin{align*}
	s^2 =& \delta^6 -  3\omega(1-\omega)x^3 y^3\delta^3 - (1+2w^2) \alpha \beta \delta^3 + 
	\tfrac14 \left( \left(3\omega(1-\omega)x^3 y^3 + (1+2w^2) \alpha \beta \right)^2\right) \\
	=& \gamma^3 - x^3 \beta^3 - y^3 \alpha^3 + \alpha^2 \beta^2 + \tfrac14\left(  -27 x^6 y^6 + 18 x^3 y^3 \alpha \beta + - 3 \alpha^2 \beta^2 \right)  \\
	=& \gamma^3 - \tfrac{27}4 x^6 y^6 - x^3 \beta^3 - y^3 \alpha^3 + \tfrac{1}4 \alpha^2 \beta^2 + \tfrac{18}4 x^3 y^3 \alpha \beta \\
	=& \gamma^3 + \tfrac{1}4 \Delta (x^3, \alpha, \beta, y^3), 
	\end{align*}
	where $\Delta$ is the discriminant function as before. 
\end{proof}

\section{A Geometric Construction of the Brauer class} \label{sect:geometric}

Fix the bijection 
\begin{align*} 
	\mathbb{A}^4(k) &\longleftrightarrow \left\{ \text{ binary cubic forms over $k$ }\right\}
\end{align*} 
that identifies point $(a,b,c,d)$ with the binary cubic form $f(u,v) =  au^3 + 3bu^2v + 3 cuv^2 + dv^3$. We will use this identification throughout the upcoming section. The discriminant function $\Delta = \Delta (x^3, \alpha, \beta, y^3)$ defines the set of nondegenerate cubic forms $U = D(\Delta) \subset \mathbb{A}^4$. The affine coordinate ring of $U$ is denoted by $R = k\left[ x^3, \alpha, \beta, y^3 \right]_\Delta$. Denote by $Z$ the center of $A$ as calculated in \Cref{thm:center-of-A}. Let $A_\Delta = A \otimes_Z Z_\Delta$. The main goal of this section is the following theorem.

\begin{thm}\label{thm:ADelta-is-Azumaya} 
	$A_\Delta$ is an Azumaya algebra over its center.
\end{thm}  

We construct the Brauer class of $A_\Delta$ geometrically. This generalizes the construction of the Clifford algebra associated to a binary form given by Kulkarni in \cite{Kulkarni:On-the-Clifford-algebra-of-a-binary-form} and  \cite{Kulkarni:the-extension-of-the-reduced-clifford-algebra-and-its-brauer-class}. Many proofs we provide in this section are similar to the construction over a field. We include them here for completion. We begin by reviewing some facts about Jacobians and \Poincare bundles. 

\subsection{The \Poincare bundle} 

Let $\mathcal{C}$ be the nonsingular relative curve over $U$ defined by 
\begin{align*}
\xymatrix@C=-0.1em{\mathcal{C} \ar[d] & = \Proj \left( \frac{R[u,v,w]}{w^3 - x^3 u^3 - \alpha u^2v - \beta uv^2 - y^3 v^3}\right) \phantom{Pr}\\ U & = \Spec(R) = \Spec \left( k\left[ x^3,\alpha,\beta, y^3\right]_\Delta\right)  }\label{eqn:defnofC}
\end{align*}
By \cite[Theorem 8.1]{Milne:Jacobian-Varieties} the Jacobian scheme $J$ of $\mathcal{C}$ over $U$ exists and its fibers are connected. Furthermore, there is a morphism of functors $\Pic^0_{\mathcal{C}/U} \ra J$ so that 
(1)
for any $U$-scheme $T$ the map $\Pic^0_{\mathcal{C}/U}(T) \ra J(T)$ is injective and 
(2) it is an isomorphism whenever $\mathcal{C} \times_U T \ra T$ admits a section. 
Since $\mathcal{C}$ is nonsingular, $J$ is given by a family of Jacobian varieties \cite[page 193]{Milne:Jacobian-Varieties} and an affine equation for $J$ is  
\begin{align*}
s^2 &= \gamma^3 - \frac{27}4 x^6 y^6 + \frac14 \alpha^2 \beta^2 + \frac{18}4 x^3y^3 \alpha \beta - x^3\beta^3 - y^3 \alpha^3\\
&= \gamma^3 + \frac{1}{4} \Delta(x^3, \alpha, \beta, y^3).
\end{align*}
Comparing the above equation to the description of the center of $A$ in \cref{eqn:elliptic-curve-eqn}, we see that $J$ is the relative elliptic curve over $U$ whose coordinate ring is isomorphic to the center of $A_\Delta$. \\

We will now review some background on the \Poincare bundle which we will construct on a cover of $U$. For more background on see \cite[Ch. 8]{Bosch:Neron-Models}. Let $S$ be any scheme and $X$ an $S$-scheme. Denote by $p:X \ra S$ the structure morphism. Assume that $p_*(\O_X) = \O_S$ holds universally, i.e. after any base change, and suppose that $X$ admits a section $s: S \ra X$. A \emph{rigidified line bundle along $s$} is a pair $(\mathcal{L}, \alpha)$, where $\mathcal{L}$ is a line bundle on $X$, and $\alpha$ is an isomorphism $\O_S \ra s^*(\mathcal{L})$. Let $(P,s)$ be the functor from the category of $S$-schemes to the category of sets that assigns to an $S$-scheme $T$ the set of isomophism classes of line bundles on $X \times_S T$ that are rigidified along the induced section $s_T: T \ra X \times_S T$. This functor is canonically isomorphic to the relative Picard functor $\Pic_{X/S}$.  

\begin{prop}[{\cite[\S 8.2, Proposition 4]{Bosch:Neron-Models}}]
	Let $p: X \ra S$ be finitely presented and flat, and let $s: S \ra X$ be a section. Assume that $p_*\O_X =\O_S$ holds universally and that $\Pic_{X/S}$ is representable by a scheme. The \Poincare bundle $\mathcal{P}$ for $(X/S,s)$ has the following universal property: 
	For any $S$-scheme $S'$ and for any line bundle $L'$ on $X'=X \times_S S'$ which is rigidified along the induced section $s': S' \ra X'$, there exits a unique morphism $g: S' \ra \Pic_{X/S}$ such that $L'$ is isomorphic to the pull-back of $\mathcal{P}$ under the isomorphism $1_X \times g$ as rigidified line bundles. 
\end{prop}

Note first that there exists an \etale cover of $U$ that splits $\mathcal{C}$ so that the \Poincare bundle exists by \cite[Ch. 1, Proposition 3.26]{Milne:Etale-Cohomology}. Instead of considering this cover, we first determine an open covering of $U$ so that each open set in the covering admits a degree $3$ cover that splits $\mathcal{C}$. Fix the cover $\mathcal{U} = \left\{ U_1, U_2, U_3, U_4\right\}$ of $U$, where 
\begin{itemize} 
	\item $U_1$ is defined by $x^3 \neq 0$,
	\item $U_2$ is defined by $y^3 \neq 0$,
	\item $U_3$ is defined by $x^3+ \alpha + \beta + y^3 \neq 0,$ and 
	\item $U_4$ is defined by $x^3 - \alpha + \beta -y^3 \neq 0$. 
\end{itemize} 
Since the characteristic of $k$ is different from $2$, the $U_i$ cover $U$.
	 
\begin{prop}\label{prop:cover}  
	For $1 \leq i \leq 4,$ there are Galois covers $V_i$ of $U_i$ of degree $3$ so that $$\mathcal{C}_i = \left( \mathcal{C} \times_U U_i \right) \times_{U_i} V_i = \mathcal{C} \times_U V_i$$ admits a section over $V_i$.
\end{prop} 

\begin{proof} 
Denote the third root of $x^3$ by $x$ and let $V_1$ be the cover of $U_1$ obtained by attaching $x$. Then $(x:0:x^2)$ defines a point on $\mathcal{C}_1$. Similarly, $(0:y:y^2)$ defines a point on $\mathcal{C}_2$, where $V_2$ is obtained by attaching a third root $y$ of $y^3$. Let $t$ be a third root of $x^3+ \alpha + \beta + y^3$ and let $V_3$ be the cover of $U_3$ obtained by attaching $t$. The point $(1:1:t)$ defines a point on $\mathcal{C}_3$. Similarly, $(1:-1:t)$ defines a point on  $\mathcal{C}_4$, where $V_4$ is obtained by attaching a third root $t$ of $x^3 - \alpha + \beta -y^3$. 
\end{proof} 

Let $V_{ij} = V_i \times_U V_j$ and denote $\mathcal{C}_{i} = \mathcal{C}\times_U V_i, \mathcal{C}_{ij} = \mathcal{C} \times_U V_{ij}$. By the previous proposition, the curves $\mathcal{C}_i$ admit sections and therefore the \Poincare bundles $\mathcal{P}_i$ on $\mathcal{C}_i \times \Pic^3_{\mathcal{C}_i/V_i}$ exist. Consider the sheaf of algebras $\End \left( \left(\pi_i\right)_* \mathcal{P}_i\right),$ where 
$$\pi_i : \mathcal{C}_i \times \Pic^3_{\mathcal{C}_i/V_i} \rightarrow \Pic^3_{\mathcal{C}_i/V_i}$$ 
is the projection onto the second factor. Let $q_i: V_{ij} \rightarrow V_i$ and $q_j: V_{ij} \rightarrow V_j$ be the projections. Denote the maps induced by $q_i$ on $\Pic^3_{\mathcal{C}_i/V_i}$ and $\mathcal{C}_i \times \Pic^3_{\mathcal{C}_i/V_i}$ by $q_i^J, q_i^{\mathcal{C} \times J}$, respectively. Consider the commutative diagram \begin{equation}\label{eqn:comm-diagram}
\xymatrix{
\mathcal{C}_i \times \Pic^3_{\mathcal{C}_i/V_i} \ar[d]^{\pi_i} & \mathcal{C}_{ij} \times \Pic^3_{\mathcal{C}_{ij}/V_{ij}} \ar[d]^{\pi_{ij}} \ar[l]_{q_{i}^{\mathcal{C} \times J}} \ar[r]^{q_{j}^{\mathcal{C} \times J}} & \mathcal{C}_j \times \Pic^3_{\mathcal{C}_j/V_j} \ar[d]^{\pi_j} \\ 
\Pic^3_{\mathcal{C}_i/V_i} & \Pic^3_{\mathcal{C}_{ij}/V_{ij}} \ar[l]_{q_{i}^J} \ar[r]^{q_{j}^J} & \Pic^3_{\mathcal{C}_j/V_j}
},
\end{equation}
where the vertical maps are projections onto the second factor. 

\begin{lemma} 
	With notation as above, there are isomorphisms
	$$\phi_{ij}: \left( q_i^J \right)^* \End \left( \left(\pi_{i}\right)_* \mathcal{P}_{i} \right) 
	\rightarrow \left( q_j^J\right)^*  \End \left( \left(\pi_{j}\right)_* \mathcal{P}_{j} \right). $$
\end{lemma} 

\begin{proof} Remark first that $q_{i}^J$ and $q_j^J$ are flat and diagram \ref{eqn:comm-diagram} is commutative. Using \cite[Ch. III, Proposition 9.3]{Hartshorne}, we see that $\left( q_{i}^J \right)^*\left( \pi_{i}\right)_* \mathcal{P}_{i} =  \left(\pi_{ij}\right)_*  \left( q_{i}^{\mathcal{C} \times J}\right)^*  \mathcal{P}_{i}$. Therefore, 
	\begin{align*}
	\left(q_{i}^J \right)^*\End \left( \left(\pi_{i}\right)_* \mathcal{P}_{i} \right) 
		 &\cong \End \left(  \left( q_{i}^J \right)^* \left(\pi_{i}\right)_* \mathcal{P}_{i}   \right) \\
		 &\cong \End \left(  \left(\pi_{ij}\right)_*  \left( q_{i}^{\mathcal{C} \times J}\right)^*  \mathcal{P}_{i} \right) \\
		 &\cong \End \left(  \left(\pi_{ij}\right)_*  \left( q_{j}^{\mathcal{C} \times J}\right)^* \mathcal{P}_{j} \right) \\
		 &\cong \left( q_{j}^J \right)^*  \End \left( \left(\pi_{j}\right)_* \mathcal{P}_{j} \right),
		\end{align*} 
where the third equality holds by universality of the \Poincare bundle on ${\mathcal{C}}_{ij} \times \Pic^3_{\mathcal{C}_{ij}/V_{ij}}$ \cite[Ch. 8, Section 2, Propostion 4]{Bosch:Neron-Models}. 
\end{proof}
 
\begin{prop} Using the above notation for $\mathcal{P}_i$ and $\pi_i$, $\End\left( \pi_* \mathcal{P}_i \right)$ glue to a sheaf on the \etale site of $\Pic^3_{\mathcal{C}/U}$. Denote the descended Azumaya algebra by $\mathcal{A}$. 
\end{prop}

\begin{proof}
	For each triple $i,j,k$, denote by $V_{ijk} = V_i \times_U V_j \times_U V_k$ and let $p^J_{i}$ be the map induced on $\Pic^3_{\mathcal{C}\times_U V_{ijk}/V_{ijk}}$ by the map $q_{i}: {V}_{ijk} \rightarrow V_i$. Let $\phi_{ij}$ be the restrictions of the isomorphisms from the previous lemma. By \cite[\href{https://stacks.math.columbia.edu/tag/04TP}{Tag 04TP, Lemma 7.26.4}]{stacks-project}, it suffices to show that 
	$$\xymatrix{ \left( p_{i}^J\right)^*\End \left( \left(\pi_{i}\right)_* \mathcal{P}_{i} \right) \ar[rd]^{\phi_{ij}}
		\ar[rr]^{\phi_{ik}} && \left( p_{k}^J\right)^*\End \left( \left(\pi_{k}\right)_* \mathcal{P}_{k} \right)\\
		& \left( p_{j}^J \right)^*\End \left( \left(\pi_{j}\right)_* \mathcal{P}_{j} \right) \ar[ru]^{\phi_{jk}}
	}$$  commutes for all $i,j,k$. This holds true by construction of $\phi_{ij}$.
\end{proof}

We now review the construction of  the \Poincare bundle using a $\Theta$-divisor. Fix the identity section on the Jacobian $J$ of $\mathcal{C}$ and recall that the curve $\mathcal{C}$ is equipped with a relative very ample line bundle $\mathcal{O}(1)$ of degree $3$. It defines a divisor called the $\Theta$-divisor in $\Pic^3_{\mathcal{C}/U}$. This theta divisor corresponds to the zero section under the morphism 
$$\xymatrix{J = \Pic^0_{\mathcal{C}/U} \ar[r]^{\otimes \mathcal{O}(1)}&  \
\Pic^3_{\mathcal{C}/U}}.$$
Denote by $\Theta_0$ the divisor in $\Pic^0_{\mathcal{C}/U}$ corresponding to the twisted line bundle $\mathcal{O}(\Theta)(-1)$.
\begin{lemma}\label{lemma: complement of theta is affine} 
The complement of $\Theta$ in $\Pic^3_{\mathcal{C}/U}$ is affine. Denote this complement by $\Spec(S)$. 
\end{lemma}

\begin{proof}
First observe that by Lemma 1.2.3, \cite{Katz-Mazur}, $\Theta$ is a Cartier divisor over $U$ and hence $\mathcal{L} = \mathcal{O}(\Theta)$ defines an invertible sheaf on $\Pic^3_{\mathcal{C}/U}$. Next we show that $\mathcal{L}^3$ is relatively very ample. By Lemma 28.36.7, \cite{stacks-project}, it suffices to show that the canonical map $p^* \left( p\right)_* \mathcal{L}^3 \rightarrow \mathcal{L}^3$ is surjective, and that the associated map $X \rightarrow \mathbb{P}\left(\left(p \right)_* \mathcal{L}^3\right)$ is an immersion. Now by Corollory 12.9, Chapter 3, \cite{Hartshorne}, $\left( p\right)_*\mathcal{L}^3$ is locally free of rank 3 since the dimension of $H^0$ is constant and equal to $3$  in fibres by Riemann-Roch. Next the canonical map $p^* p_* \mathcal{L}^3 \rightarrow \mathcal{L}^3$ is surjective in fibers since the restriction of $\mathcal{L}^3$ is very ample in each fiber. So by Nakayama's Lemma, the canonical map $p^* \left( p\right)_* \mathcal{L}^3 \rightarrow \mathcal{L}^3$ is surjective. Finally by very ampleness of  $\mathcal{L}^3$ along the fibres, the induced morphism  $X \rightarrow \mathbb{P}\left(\left( p\right)_* \mathcal{L}^3\right)$ is an immersion along the fibers and hence it is an immersion itself.

Now since $U$ is itself affine, it follows from Lemma 29.39.4, \cite{stacks-project} that some multiple of $\Theta$ is a very ample divisor in $\Pic^3_{\mathcal{C}/U}$. In particular, the complement of $\Theta$ in $\Pic^3_{\mathcal{C}/U}$ is affine.
\end{proof}

\subsection{Construction of an isomorphism $\mathcal{A} \rightarrow A_\Delta$}

We want to show that the Azumaya algebra $\mathcal{A}$ given by the descent of  $\End\left(\left( \pi_i\right)_*(\mathcal{P}_i)\right)$  and the binary cubic generic Clifford algebra $A$, constructed algebraically using generators and relations, agree on the complement of the $\Theta$ divisor. Recall that we denote by $\mathcal{U}= \left\{ U_1, \ldots, U_4\right\}$ a cover of $U$ so that there are $V_i \rightarrow U_i$ Galois covers of degree $3$ and $\mathcal{C}_i = \mathcal{C} \times_U V_i$ splits over $V_i$ (see \Cref{prop:cover}). We will first construct the morphism over $V_i$ and then descend it to a morphism over each $U_i$. 

\begin{prop}\label{prop:DefineP}
	The push-forward $\left( \pi_i\right)_* \mathcal{P}_i$ of the \Poincare bundle is locally free of rank $3$ over $\Pic^3_{\mathcal{C}_i/V_i}$, where $\pi: \mathcal{C}_i \times \Pic^3_{\mathcal{C}_i/V_i} \rightarrow  \Pic^3_{\mathcal{C}_i/V_i}$ is the projection. In particular, for the complement $\Spec(S_i)$ of $\Theta$ in $\Pic^3_{\mathcal{C}_i/U_i}$, there is a projective $S_i$-module $P_i$ so that $\tilde{P}_i \cong \left. \pi_*\mathcal{P}_i \right|_{\Spec(S_i)}.$
\end{prop}

\begin{proof}
Let $t \in \Pic^3_{\mathcal{C}_i/V_i}$ be a closed point and denote its residue field by $k_t$. Then there is some closed point $f \in V_i$ with residue field $k_f$ so that $t$ lies in the fiber over $f$. Consider the Cartesian square 
$$\xymatrix{
\mathcal{C}_f \times \Pic^3_{\mathcal{C}_f/k_f} \ar[r]^i\ar[d]^{\pi_f} & \mathcal{C}_i \times \Pic^3_{\mathcal{C}_i/V_i} \ar[d]^{\pi_i} \\
\Pic^3_{\mathcal{C}_f/ k_f} \ar[r] & \Pic^3_{\mathcal{C}_i/V_i}
}.$$
By the universal property of the \Poincare bundle, the pullback $\mathcal{P}_f = i^* \mathcal{P}_i$ is the \Poincare bundle on $C_f \times \Pic^3_{\mathcal{C}_f/k_f}$. In particular, it is locally free of degree $3$.  Then $(\pi_f)*\mathcal{P}_f$ is torsion free, as it is the push-forward of a torsion free sheaf. Since $\Pic^3_{\mathcal{C}_f/k_f}$ is a curve, we deduce that $(\pi_f)*\mathcal{P}_f$ is locally free of rank $3$. We calculate that $$\dim_{k_t}  \left( \left( \pi_i\right)*\mathcal{P}_i \otimes k_f \right) = \dim_{k_f}  \left( (\pi_f)_*\mathcal{P}_f \otimes k_t\right) = 3.$$ 
Finally, by \cite[Page 51, Lemma 1]{Mumford:Abelian-Varieties} $\left( \pi_i \right)_*\mathcal{P}_i$ is locally free of rank $3$ on $\Pic^3_{\mathcal{C}_i/V_i}$. 
%
\end{proof} 

Recall that we denote by $\Spec S$ the affine complement of $\Theta$ in $\Pic^3_{\mathcal{C}/U}$. Let $S_i$ be defined by $\Spec S_i = \Spec S \times_U V_i$. Consider the sequence of morphisms 

$$\mathcal{C}_i\times_{V_i} \Spec S_i \overset{q_{S_i}}\longrightarrow \mathbb{P}^1_{V_i} \times_{V_i} \Spec S_i \overset{p_{S_i}}\longrightarrow \Spec S_i.$$
Then $\pi_i= p_{S_i} \circ q_{S_i}$ is the projection onto the second factor. Denote by $\mathcal{F}$ the coherent sheaf $(q_{S_i})_*\mathcal{P}_i$ on $\mathbb{P}^1_{V_i} \times_{V_i} \Spec S_i$. 

\begin{prop}
	The natural morphism $u: p^*_{S_i} (p_{S_i})^* \mathcal{F} \ra \mathcal{F}$ is an isomorphism. 
\end{prop}

\begin{proof} 
	By the previous proposition $(p_{S_i})_* \mathcal{F} = \left( \pi_i \right)_*\mathcal{P}_i$ is a locally free sheaf of rank $3$ and therefore the sheaf $p^*_{S_i} (p_{S_i})^* \mathcal{F}$ is locally free of rank $3$ as well. Let $y$ be a closed point in $\mathbb{P}^1_{S_i}$. Then there exists some $f \in V_i$ such that $y$ is a closed point in $S_f = S \otimes k(f)$ and thus by the proof of \cite[Proposition 3.8]{Kulkarni:On-the-Clifford-algebra-of-a-binary-form} the dimension of $\mathcal{F} \otimes k(y)$ is $3$. By upper semicontinuity and using that $\mathbb{P}^1_{S_i}$ is of finite type over $V_i$ we see that $\mathcal{F}$ is a locally free sheaf of rank $3$.
	By \cite[Proposition 3.8]{Kulkarni:On-the-Clifford-algebra-of-a-binary-form} the morphism $u$ is an isomorphism on all closed points. Now by \cite[Corollary 0.5.5.7]{Grothendieck:Technique-de-descente-et-theoremes-dexistence-en-geometrie-II} it follows that $u$ is surjective. Since $u$ is a map between locally free sheaves of rank $3$ it follows that $u$ is an isomorphism. 
\end{proof} 

\begin{cor} 
	The $S_i[u,v]$-module $\bigoplus_j H^0 \left( \mathbb{P}^1_{S_i}, \mathcal{F}(j) \right)$ is graded isomorphic to the module $M:= P_i \otimes_{S_i} S_i[u,v]$, where $P_i$ is as in \Cref{prop:DefineP}. 
\end{cor}

\begin{proof}
	By the projection formula and using the fact that $p_{S_i}^* (p_{S_i})_* \mathcal{F} \cong \mathcal{F}$ we see that 
	$$(p_{S_i})_* \mathcal{F}(j) \cong (p_{S_i})_* \mathcal{F} \otimes (p_{S_i}) \left( \O_{\mathbb{P}^1_{S_i}} (j) \right)$$
	Now since $\mathbb{P}^1_{S_i}$ is projective over $S_i$, the module $\bigoplus_j H^0 \left( \mathbb{P}^1_{S_i}, \mathcal{F}(j) \right)$ is obtained by 
	$$\bigoplus_j H^0 \left(\Spec S_i, (p_{S_i})_*\mathcal{F} \times (p_{S_i}) (\O_{\mathbb{P}^1_{S_i}}(j))\right) 
	\cong P_i \otimes_{S_i} \left( \bigoplus_j (p_{S_i})_* (\O_{\mathbb{P}^1_{S_i}}(j))\right) .
	$$
	The statement follows from the computation of the module associated to the structure sheaf of the projective space over $S_i$. 
\end{proof}

We are now ready to describe the desired isomorphism over $S_i$. Recall that $U = \Spec R = k \left[ x^3, \alpha, \beta, y^3 \right]_\Delta$, and $V_i = \Spec R_i$. 
\begin{prop}\label{prop:morphism before descent} There exists an algebra homomorphism 
	$$\psi_i:  A_\Delta \otimes_R R_i \ra \End_{S_i}\left( P_i\right)$$
	with $P_i$ is as in \Cref{prop:DefineP}. 
\end{prop} 

\begin{proof}
	Analogous to the proof in \cite[Proposition 4.2]{Kulkarni:On-the-Clifford-algebra-of-a-binary-form}, we will find elements in $\End_{S_i}\left(P_i\right)$ that satisfy the relations of the Clifford algebra. Denote by $P^j_i$ the $j$-th graded piece of the module $M = P_i \otimes_{S_i} S_i[u,v]$. Note that $P^1_i \cong P^0_i \oplus P^0_i$ with generators $u$ and $v$. We identify $u$ and $v$ with the corresponding elements in $\Hom\left(P^0_i,P^1_i\right)$ given by left multiplication by $u$ and $v$, respectively. Then a straightforward calculation shows that  
	\begin{equation}\label{eq:Hom decomposes} \Hom\left(P^0_i,P^1_i\right) = u \circ \End\left(P^0_i\right) \oplus v \circ \End\left(P^0_i\right).\end{equation} 
	Since $\mathcal{P}_i$ is the \Poincare bundle on $\mathcal{C}_i \times \Jac \mathcal{C}_i$ we see that $M$ is also a graded module over $S_i[u,v,w]/(w^3 - x^3u^2v - \alpha u^2v - \beta uv^2 - y^3 v^3)$, where the degree of $w$ is one. Identify $w$ with the element in $\Hom\left(P^0_i,P^1_i\right)$ given by left multiplication. By the decomposition in (\ref{eq:Hom decomposes}) there are elements $\alpha^i_u, \alpha^i_v \in \End\left(P_i\right)$ so that 
	\begin{equation} \label{eqn:alphau} w = u \circ \alpha^i_u + v \circ \alpha^i_v.\end{equation} 
	Then the equation
	$$\left( u \circ \alpha^i_u + v \circ \alpha^i_v \right)^d = w^d = x^3u^3 + \alpha u^2v + \beta uv^2 + y^3 v^3$$
	holds true in $\Hom\left(P^0_i,P^d_i\right)$. Since $\Hom\left(P^0_i,P^d_i\right)$ is a free $\End\left(P^0_i,P^0_i\right)$-module, the elements $\alpha^i_u$ and $\alpha^i_v$ satisfy the relations of the Clifford algebra. Define the morphism $\psi_i$ by setting $\psi_i(x) = \alpha^i_u, \psi_i(y) = \alpha^i_v$. 
\end{proof}

It remains to prove that the morphism in the previous proposition descends to an isomorphism over $U$ via \etale descent, which will imply the following statement. 
\begin{prop}\label{prop:morphism after descent} 
	There exists an algebra isomorphism $\psi: A_\Delta \ra \mathcal{A}\otimes S$. 
\end{prop}

\begin{proof} 
	Denote $S_{ij} = S_i \otimes_{S} S_j$. This is the affine complement of the $\Theta$-divisor in $\Jac \mathcal{C}_{ij}$, where $\mathcal{C}_{ij} = \mathcal{C} \times_U V_{ij}$ as before. Denote the module 
	$$P_{ij} = P_i \otimes_{S_i} {S_{ij}} \cong P_j \otimes_{S_j} S_{ij}.$$
	Let $\alpha_u^i$, $\alpha_v^i$, $\alpha_u^j$, and $\alpha_v^j$ be chosen as in the previous proposition. Since they both satisfy \cref{eqn:alphau}, we get on the intersection that  
	\begin{equation*}
	u \circ \left. \alpha_u^{i}\right|_{P_{ij}} + v \circ \left. \alpha_v^{i}\right|_{P_{ij}}  =	w = 
	u \circ \left. \alpha_u^{j}\right|_{P_{ij}} + v \circ \left. \alpha_v^{j}\right|_{P_{ij}}.
	\end{equation*}
	By \cref{eq:Hom decomposes}, the choice of such elements is unique. We conclude that $\psi_i$ and $\psi_j$ agree on the intersection of their domains. \\

	It remains to show that $\psi$ is an isomorphism, which is a property that is stable under descent. Thus it is enough to show that $\psi_i$ is an isomorphism. Note that by \cite[II Proposition 1.1]{Hartshorne} it is further enough to check that $\psi_i$ is an isomorphism on stalks, which is true by \cite[Theorem 6.2]{Kulkarni:On-the-Clifford-algebra-of-a-binary-form}. Finally, the isomorphisms $\psi_i$ glue to an isomorphism $\psi: A_\Delta \rightarrow \mathcal{A} \otimes S$. 
\end{proof} 

\subsection{The Brauer class is nontrivial}
In this section, we prove that the Brauer class associated to the binary cubic generic Clifford algebra is not trivial over any base field $k$. Recall that the Brauer class associated to a binary cubic form $f$ is trivial if and only the curve $\mathcal{C}_f$ admits a $k$-rational point \cite{Haile:WhenistheCliffordAlgebraofabinarycubicformsplit}. We will first exploit this relationship for the generic Clifford algebra at the generic point. 

\begin{prop}\label{prop:Ck(U)-does-not-have-a-point}The curve $\mathcal{C}_{k(U)}$ does not have a rational point over $k(U)$. In particular, the Brauer class of $A_\Delta \otimes k(U)$ is not trivial for any field $k$. \end{prop} 

\begin{proof}Let $a,b,c,d$ be algebraically independent variables. Then the function field of $U$ is $k(U) \cong k(a,b,c,d)$.  Suppose for a proof by contradiction that there are $f,g,h \in k(a,b,c,d)$ not all zero such that 
	$$h^3 = a f^3 + bf^2 g + c fg^2 + dg^3.$$
	After clearing denominators we may assume without loss of generality that $f,g,h \in k[a,b,c,d]$.\\

	Modulo $(b,c)$ we get the equation $h^3 = a f^3 + d g^3$ in $k[a,d]$. We will prove that this has no nontrivial solution. Note that $f=0$ would imply that $g=h=0$ by degree considerations. Assume that $f$ is nonzero of minimal total degree. Now
	\begin{align*} 
	h^3 \equiv af^3 \mod (d)\\
	h^3 \equiv dg^3 \mod (a)
	\end{align*}
	By degree considerations we see that $f\equiv h \equiv 0 \mod (d)$, and $g \equiv h \equiv 0 \mod (a)$. Thus there are $h_1, f_1, g_1 \in k[a,d]$ such that 
	$ h = ad h_1, f = d f_1, g = a g_1$. After division by $ad$ the above equation becomes
	$$ a^2 d^2 h_1^3 =  d^2 f_1^3 + a^2 g_1^3$$
	Now again taking modulo $a$ and $d$ we see that there are $f_2,g_2 \in k[a,d]$ so that $f_1 = a f_2$ and $ g_1 = d g_2$. Thus 
	$h_1^3 = a f_2^3 + d g_2^3,$ where $f_2$ is of smaller degree than in the original equation. We conclude that $f,g,h \equiv 0 \mod (b,c)$. \\

	Further, taking the original equation modulo $c$ gives $h^3 = a f^3 + b f^2 g + d g^3$ in $k[a,b,d]$. As before, we assume that $f$ is nonzero of minimal degree. Modulo $b$ we get 
	$$h^3 \equiv a f^3 + d g^3 \mod (b)$$
	which does not have a non-trivial solution by the previous. Thus 
	$f = b f_1, g = bg_1,$ and $h = bh_1$ for some $f_1, g_1, h_1 \in k[a,b,d]$. Therefore $h_1^3 = a f_1^3 + b f_1^2 g_1^2 + d g_1^3$, which is of smaller degree. \\
	
	Finally, consider the original equation again and assume that $f$ is nonzero of minimal degree. Taking the equation modulo $c$ we get
	$$h^3 \equiv a f^3 + b f^2 g + dg^3 \mod (c).$$
	By the previous this does not have any nontrivial solution. Therefore $f = cf_1, g = c g_1,$ and $h = c h_1$ for some $f_1, g_1, h_1 \in k[a,b,c,d]$. Thus $h_1^3 = a f_1^3 + b f_1^2 g_1 + c f_1 g_1^2 + d g_1^3$, which is an equation of smaller degree; a contradiction. 
\end{proof}

\begin{thm}\label{thm:Brauer-class-nontrivial}The Brauer class of $A_\Delta$ is not trivial. \end{thm} 

\begin{proof} 
	The map $\Br(J) \ra \Br\left( J_{k(U)}\right)$ is injective \cite[Corollary 1.8]{Grothendieck:Le-groupe-de-Brauer-II}. Thus it will be enough to show that the image of $A$ under the above map is nontrivial, which follows from \cite{Haile:WhenistheCliffordAlgebraofabinarycubicformsplit} and the previous proposition. 
\end{proof} 

\subsection{$A$ is a maximal order over its center}

In this section, we show that even though $A$ is not Azumaya, it is a maximal order over its center. To that end, we show that $A$ is reflexive and a maximal order at every stalk. 

\begin{lemma}\label{lem:A-is-a-reflexive-Z-module}$A$ is a reflexive $Z$-module. \end{lemma} 

\begin{proof}
	The center $Z$ of $A$ is a free $\mathcal{S}$-module with generators $1$ and $s$, where $s$ is as in the proof of \Cref{thm:center-of-A}. Therefore, the $Z$-module $\Hom_{\mathcal{S}}\left( Z,\mathcal{S} \right)$ is free of rank one with generator $\rho$ defined by $\rho(1) = 0$ and $\rho(t) = 1$. Therefore 
	$$\Ext^{i}_{\mathcal{S}}\left(Z, \mathcal{S}\right) \cong \begin{cases} Z & \text{if } i=0 \\
		0 & \text{else} \end{cases}$$ 
	as $Z$-modules. The change of rings spectral sequence for $\Ext$ is 
	$$\Ext_Z^i \left( A, \Ext_{\mathcal{S}}^j (Z, \mathcal{S}) \right) \Rightarrow\Ext_{\mathcal{S}}^{i+j} (A,\mathcal{S}).$$
	Degeneration implies that 
	$$\Ext^i_{Z}\left(A,\Ext^0_S\left( Z,\mathcal{S} \right) \right) \cong \Ext^i_Z \left( A, Z \right) \cong \Ext^i_{\mathcal{S}} (A, \mathcal{S})$$ for all $i\geq 0$. The latter vanishes for all $i>0$ since $A$ is free of rank 18 over $\mathcal{S}$ by \Cref{thm:basisofA}. Thus $\Ext^i_{Z}(A,Z) =0$ for all $i>0$ and so $A$ is projective as a $Z$-module. Hence $A$ is also reflexive as a $Z$-module.  
\end{proof} 

We have seen that $A_\Delta$ is an Azumaya algebra \cref{thm:ADelta-is-Azumaya} and therefore it is also a maximal order away from the discriminant. It remains to 
prove that $A_{(\Delta)}$ is also a maximal order. A more explicit description of $A_\Delta \otimes k(U)$ will be useful for this proof. 
\begin{lemma} 
	Let $\epsilon = xyx + \omega x^2 y+ \omega^2  yx^2 \in A$. Then	
	$$\epsilon x  =  \omega x \epsilon \text{ and } \epsilon y = \omega y \epsilon + \gamma - \omega \gamma$$
	and $\epsilon^3$ is central in $A$. This implies that over the generic point, $A_\Delta \otimes k(U)$ coincides with the symbol algebra $\left( x^3, \epsilon^3\right)_{3,Z \otimes k(U)}$. 
\end{lemma} 

\begin{proof} We need to show that $\epsilon^3$ is central in $A$. Let us first see how $\epsilon$ interacts with $x$ and $y$:   
	$$\begin{gathered} \epsilon x =  xyx^2 + \omega x^2 yx + \omega^2  yx^3  = x \left( yx^2 + \omega xyx + \omega^2 x^2 y \right) = \omega x \epsilon,\\
 	\epsilon y = (xy)^2 + \omega x^2 y^2 + \omega^2 yx^2 y = \gamma + y^2x^2 - \omega \gamma + \omega (yx)^2 + \omega^2 yx^2y = \omega y \epsilon + \gamma - \omega \gamma.\end{gathered} $$ This implies that $\epsilon^3 x = x \epsilon^3$ and 
	$$	\epsilon^3 y = \omega \epsilon^2 y \epsilon + \gamma - \omega \gamma = \omega^2 \epsilon y \epsilon^2 + \omega \gamma - \omega^2 \gamma + \gamma - \omega \gamma = y \epsilon^3 + \omega^2 \gamma - \gamma - \omega^2 \gamma + \gamma = y \epsilon^3.$$ Thus $\epsilon^3$ fixes both $x$ and $y$ and it is therefore central in $A$.  
\end{proof} 

\begin{prop}\label{prop:A-otimes-ZDelta-is-a-maximal-order}
	$A \otimes_Z Z_{(\Delta)}$ is a maximal order. 
\end{prop}

\begin{proof} 
	By \cite[Chapter 1, Example 1.4]{Gelfand-Kapranov-Zelevinski:Discriminats-resultants-and-multidimensional-determinants} the discriminant $$\Delta = -27 x^6 y^6 +18 x^3y^3 \alpha\beta -4x3 \beta^3-4y^3\alpha^3$$ is prime over the center $Z$. By degree consideration $\epsilon^3$ does not lie in the ideal generated by $\Delta$. By Krull's principal ideal theorem the ideal $(\Delta)$ is minimal and therefore $Z_{(\Delta)}$ is a discrete valuation ring. \\
	
	Let $G$ be the Galois group of $K$ over the quotient field of $Z_{(\Delta)}$ obtained by adjoining $x$. Then $G$ is cyclic of degree $3$ with generator $\sigma$. Furthermore, the integral closure of $Z_{(\Delta)}$ in $K$ is $Z_{(\Delta)}[x]$. Define an element $[f]$ in $H^2\left(G,\left(Z_{(\Delta)}\right)^\times \right)$ by 
	$$f: G \times G \ra \left(Z_{(\Delta)}\right)^\times: \left(\sigma^i, \sigma^j\right)\mapsto \begin{cases} \epsilon^3 & i+j \geq 3\\ 1 & i+j < 3 \end{cases}.$$
	Consider the crossed product algebra $\Delta(f,Z_{(\Delta)}[x],G)$, which is isomorphic to the symbol algebra $(x^3, \epsilon^3)_{Z_{(\Delta)}} = Z_{(\Delta)}\left< x,\epsilon \right> \subseteq A \otimes Z_{(\Delta)}$. Let $\Gamma$ be the maximal order containing the crossed product algebra $\Delta(f,Z_{(\Delta)}[x],G)$. Note that the inertia subgroup is $G_I = G$ and the conductor group of $\Delta(f,Z_{(\Delta)}[x],G)$ is trivial. By \cite[Proposition 3.4]{Williamson:Crossed-products-and-hereditary-orders} we deduce that $$Z_{(\Delta)}[x] = \left\{ s \in Z_{(\Delta)}[x]: s \epsilon \subseteq\Delta(f,Z_{(\Delta)}[x],G)\right\}.$$ In particular, $\Gamma  \subseteq \Delta(f,Z_{(\Delta)}[x],G) = A \otimes Z_{(\Delta)}$ and therefore $A \otimes Z_{(\Delta)}$ is a maximal order. 
\end{proof} 

\begin{thm}$A$ is a maximal order over its center. \end{thm} 

\begin{proof}
	We proved in \Cref{lem:A-is-a-reflexive-Z-module} that $A$ is a reflexive $Z$-module. Let $\mathfrak{p}$ be a minimal prime in $Z$. If $\mathfrak{p} = (\Delta)$, then $A \otimes Z_{(\Delta)}$ is a maximal order by \cref{prop:A-otimes-ZDelta-is-a-maximal-order}. Furthermore, if $\mathfrak{p} \neq (\Delta)$, then $A \otimes Z_{\mathfrak{p}}$ is Azumaya and therefore in particular, it is a maximal oder. Using \cite[Theorem 1.5]{Auslander-Goldman:Maximal-Orders}, we deduce that $A$ is a maximal order over its center. 	
\end{proof}
\section{The Moduli Problem}\label{sect:moduli} 

The main statement of this section is the following theorem which gives a correspondence between the moduli stack $\left[U/GL_2\right]$ and pairs $(E,\alpha)$, where $E$ is an elliptic curve of $j$-invariant $0$, and $\alpha$ is a Brauer class on $E$ that is invariant under complex multiplication. 

\begin{thm}\label{thm:Moduli-space} 
	Consider the set of pairs $\left( E, \alpha\right),$ where $E$ is an elliptic curve over $k$ of $j$-ivariant $0$,  and $\alpha$ is a $3$-torsion Brauer class in $\Br(E)/\Br_0(E)$ invariant under complex multiplication $\theta$. Identify two pairs $(E,\alpha)$ and $(E', \alpha')$ if there is an isomorphism $\phi: E \rightarrow E'$ and the induced isomorphism $\Br(E') \rightarrow \Br(E)$ identifies $\alpha$ and $\alpha'$. There is a bijection between the set of equivalence classes of pairs $(E,\alpha)$ and the set of $\GL_2(k)$-orbits of nondegenerate binary cubic forms. 
\end{thm} 

This correspondence is given as follows: Let $f$ be a nondegenerate binary cubic form. The Jacobian of the genus one curve defined by $w^3-f(u,v)$ is a $j$-invariant $0$ elliptic curve $E$. As $E$ has $j$-invariant $0$, it is equipped with an order $3$ automorphism $\theta$. We show in this section that the Clifford algebra associated to $f$ defines a $\theta$-invariant $3$-torsion Brauer class $\alpha$ on $E$. On the other hand, let $(E,\alpha)$ be as in the statement and consider the short exact sequence induced by the Hochschild-Serre spectral sequence 
\begin{equation}\label{eq:SES}\xymatrix{0 \ar[r] & \Br(k) \ar[r] & \Br(E) \ar[r]^{\kappa\phantom{klkjl}} & H^1\left( k, E(\kbar)\right) \ar[r] & 0  }.\end{equation} 
We prove in this section that the image of $\alpha$ under $\kappa$ can be described as a $k$-torsor under $E(\kbar)$ with equation $w^3-f(u,v)$, where $f(u,v)$ is a nondegenerate binary cubic form. 

We proved in \Cref{thm1} that the  binary cubic generic Clifford algebra $A$ defines a Brauer class $\alpha_0$ over a relative elliptic curve $J$ over $U$. We will equip $J$ with a $\GL_2$-action that is compatible with the action on $U$ and show that the Brauer class $\alpha_0$ descends to the quotient. 

\subsection{The Action of $\GL_2$ on Binary Cubic Forms}\label{sec:Action-on-U} 
Recall that $U = \Spec R_\Delta$, where $R= k \left[ x^3, \alpha, \beta, y^3\right]$, is identified with the space of nondegenerate binary cubic forms. Let $\mathfrak{g} = \begin{pmatrix} a& b \\ c &d \end{pmatrix} \in \GL_2(k)$ and $r \in R$. Define the action on $R$ by acting linearly on the coordinates $x$ and $y$, i.e. 
$\mathfrak{g}.x = ax+ cy$ and 
$\mathfrak{g}.y = bx+ dy$. Direct computation shows that the action of $\mathfrak{g}$ on the basis $\left\{x^3,\alpha,\beta,y^3\right\}$ of $R$ is given by the matrix 
\begin{align}
\begin{pmatrix} 
a^3 & 3a^2b & 3ab^2 & b^3 \\
a^2c & a(ad+2bc) & b(bc+2ad) & b^2 d\\
ac^2 & c(bc+2ad) & d(ad+2bc) & bd^2 \\
c^3 & 3c^2d & 3cd^2 & d^3 
\end{pmatrix}\label{matrixaction}
\end{align}

\begin{lemma}The above action fixes the ideal $\left(\Delta\right)$. In particular, the action is well-defined on $U$. 
\end{lemma} 

\begin{proof} 
	Let $\mathfrak{g} = \begin{pmatrix} a& b \\ c &d \end{pmatrix} \in \GL_2(k)$. By a direct computation 
	\begin{align*} 
	\mathfrak{g}.\Delta(x^3, \alpha, \beta, y^3) &= \mathfrak{g}. \left( -27 x^6y^6 + \alpha^2 \beta^2 + 18 x^3 \alpha \beta y^3 -4x^3 \beta^3 -4\alpha^3 y\right)\\
	&=  \left( \det(\mathfrak{g})\right)^6 \Delta(x^3, \alpha, \beta, y^3) 
	\end{align*}
	This implies that $\mathfrak{g}$ takes $(\Delta)$ to $(\Delta)$. 
\end{proof}  

We now proceed to compute the stabilizer of a point $f \in U$ under the above action of $\GL_2(k)$. Suppose first that $f(u,v) = p u^3 + r v^3$ is diagonal over $k$ and let $\mathfrak{g} = \begin{pmatrix} a&b \\ c&d\end{pmatrix} \in \stab_{GL_2(k)} (f)$. In this case,
\begin{equation}\label{eqn:stabilizer}
a^2c p + b^2 d r = ac^2 p + bd^2 r = 0.
\end{equation}
As $f$ is a nondegenerate binary cubic form, both $p$ and $r$ are nonzero. Assume first that all $a,b,c,$ and $d$ are nonzero, then the  above implies that 
$\frac{b^2 d}{a^2 c} = \frac{bd^2}{ac^2}$ and hence $\frac{b}{a} = \frac{d}{c},$ which is impossible as $\mathfrak{g}$ has nonzero determinant. Thus at least one of $a,b,c,$ and $d$ is zero. If $a =0$, then both $b$ and $c$ are nonzero. By \cref{eqn:stabilizer}, we see that $d=0$. Further as $\mathfrak{g}$ fixes $f$, in this case $b^3 r = p$ and $c^3 p = r$. This implies that $b^3 = \frac{1}{c^3} = \frac{p}{r}$. Similarly if $b=0$, then $c=0$ and $a^3 = d^3 = 1$. \\

Summarizing, let $f$ be the diagonal form $f(u,v) = pu^3 + rv^3$. Assume first that there is no $\lambda \in k$ with $\lambda^3 = \frac{p}{r}$, then 
$$\stab_{\GL_2(k)} (f) = \left\{ \begin{pmatrix} 
\omega^i & 0 \\
0 & \omega^j \end{pmatrix}: i,j \in 0,1,2 \right\}
\cong \ZZ/3\ZZ \times \ZZ/3\ZZ.$$
If there is a $\lambda \in k$ with $\lambda^3 = \frac{p}{r}$, then the stabilizer of $f$ is 
\begin{align*} 
	\stab_{\GL_2(k)} (f) &= \left\{ \begin{pmatrix} 
\omega^i & 0 \\
0 & \omega^j \end{pmatrix}: i,j \in 0,1,2 \right\}
\cup 
\left\{
\begin{pmatrix} 
0 & \omega^i \lambda \\
\frac{\omega^j}{\lambda} & 0 
\end{pmatrix} 
: i,j \in 0,1,2 \right\}\\ 
&\cong \left( \ZZ/3\ZZ \times \ZZ/3\ZZ \right) \rtimes \ZZ/2\ZZ.
\end{align*} 
Remark in particular, that in both cases the stabilizer of $f$ is finite. \\

If $f$ is a nondegenerate binary cubic form, that is not necessarily diagonal, then there is some finite extension $F$ of $k$ so that there is some $\mathfrak{h} \in \GL_2(F)$ with $\mathfrak{h}.f$ diagonal. Now 
$$\stab_{\GL_2(k)} (f) \subset \stab_{\GL_2(F)} (f) \cong \stab_{\GL_2(F)} \mathfrak{h}.f,$$ which is finite. Hence the stabilizer for every $f \in U$ is finite and we deduce the following proposition. 

\begin{prop}\label{prop:Deligne} 
	The quotient stack $[U/\GL_2]$ is a Deligne-Mumford stack.
\end{prop} 

We now proceed to show that the orbits of $U$ with respect to this action of $\GL_2$ correspond to isomorphism classes of Severi-Brauer varieties. 


\subsection{Isomorphisms of Severi-Brauer Varieties} 
\label{sec:IsoClifford}

Let $f$ and $g$ be nondegenerate binary cubic forms and assume that their Clifford algebras $A_f$ and $A_g$ are isomorphic. This induces an isomorphism $\phi: C_f \rightarrow C_g$ of the associated Brauer-Severi varieties. In this part, we show the following statement.

\begin{prop}\label{prop:iso-SBV}
	Let $f$ and $g$ be nondegenerate binary cubic forms. Their Clifford algebras $A_f$ and $A_g$ are isomorphic if and only if there is some $\mathfrak{g} \in \GL_2(k)$ so that $\mathfrak{g}.f = g$, where the action is given by a linear change of variables. 
\end{prop} 

The main tool in our proof is the following lemma. 
\begin{lemma}\label{thm:iso-of-Clifford-alg}
	If $\phi^*\left( \mathcal{O}_{C_g} \right) = \mathcal{O}_{C_f}$, then $\phi$ extends to an isomorphism $\tilde{\phi}$ on $\mathbb{P}^2$. In particular, $f$ and $g$ are related by a linear change of variables. 
\end{lemma}

\begin{proof}
	Define $\tilde{\phi}$ via the composition 
	$$\xymatrix{H^0\left( \mathbb{P}^2, \mathcal{O}(1)\right) 
		\ar[r]^\sim & H^0\left( C_f, \mathcal{O}(1)\right)
		\ar[r]^{(\phi^*)^{-1}}_\sim & H^0\left( C_g, \mathcal{O}(1)\right)  
		\ar[r]^\sim & H^0\left( \mathbb{P}^2, \mathcal{O}(1)\right)}.$$
	This proves the statement. 
\end{proof}

Suppose from now on that $\phi^*\left( \mathcal{O}_{C_g} \right) \neq \mathcal{O}_{C_f}$. We will show in the following that this is not possible and deduce \Cref{prop:iso-SBV} from \Cref{thm:iso-of-Clifford-alg}. Consider the composition 
$$\xymatrix{ C_f \ar[r]^\phi& C_g \ar[r]^{\eta} & C_g \ar[r]^{\phi^{-1}}& C_f },$$
where $\eta: C_g \rightarrow C_g$ is the map defined by $w \mapsto \omega w$. Denote the induced map on the Jacobian $E_f$ by $\mu$ and note that the line bundle $\mathcal{L} = \phi^*\left( \mathcal{O}_{C_g}(1)\right)$. Note that $\mu$ fixes the bundle $\mathcal{L}(-1) = \phi^* \mathcal{O} = \mathcal{O}$ by the commutative diagram below
$$\xymatrix{  
	E_f \ar[r]^\mu & E_f & \mathcal{L}(-1) \\
	\Pic^3 C_f \ar[u] \ar[r]^\sigma & \Pic^3 C_f \ar[u] & \mathcal{L} \ar@{|->}[u]
}.$$
Therefore $\mu$ is a group homomorphism on $E_f$. 

\begin{lemma} 
	$\sigma$ is the map induced by $w \mapsto \omega^i w$ on $C_f$ for $i=1$ or $i=2$. In particular, $\sigma$  is defined over $k$. 
\end{lemma} 

\begin{proof} 
	The map $\sigma$ fixes exactly $3$ line bundles $\mathcal{O}(1), \mathcal{L}, $ and $\mathcal{L}^{2}(-1)$ in $\Pic^3C_f$. Therefore $\sigma$ is an order $3$ automorphism of $E_f$.  
	Since $E_f$ has $j$-invariant $0$, this autormorphism of $E_f$ is unique \cite[III Theorem 10.1]{silverman}. Hence $\sigma$ is the map induced by $w \mapsto \omega^i w$ on $C_f$ for $i=1$ or $i=2$. We may assume for the remainder of this proof, that $i=1$.\\
	
	To show that $\sigma$ is defined over $k$, let $\tau \in \Gal(\kbar/k)$ and consider the commutative diagram 
	$$\xymatrix{ C_f \ar[d]^\eta\ar[r]^\tau & C_f \ar[d]^\eta\\
		C_f \ar[r]^\tau & C_f },$$
	where $\eta$ is the map defined by $w \mapsto \omega w$. The commutativity of the above diagram implies that the action of $\tau$ on $\Pic^3 C_f$ commutes with $\sigma$ as well. Hence $\sigma$ is defined as a morphism of schemes over $k$.
\end{proof} 

By the above, there are actions of $\sigma$ on $H^{0}(C_f, \mathcal{O}(1))$ and $H^{0}(C_f,\mathcal{L})$. The former can be described by the matrix 
$$\begin{pmatrix} 
1 & 0 & 0 \\
0 & 1 & 0 \\
0 & 0 & \omega
\end{pmatrix} \in \PGL_2(k).$$
Suppose that the latter action is given by the action of a matrix $B$ in Jordan-canonical form with respect to a basis $r,s,t$. The eigenspaces of the action defined by $\mathcal{O}(1)$ are as follows: 
\begin{itemize} 
	\item The eigenspace of the eigenvalue $1$ is spanned by $u^3, u^2v, uv^2, v^3, w^3$. As there is a relation between these objects in $C_f$, this eigenspace is of dimension $5$. 
	\item The eigenspace of the eigenvalue $\omega$ is spanned by $u^2w, v^2w, uvw$.  This eigenspace is of dimension $3$. 
	\item The eigenspace of the eigenvalue $\omega^2$ is spanned by 
	$uw^2, vw^2$. This eigenspace is of dimension $2$. 
\end{itemize} 

\begin{lemma}
	$B$ is diagonal and we may assume that its eigenvalues are contained in the set $\{1,\omega,\omega^2\}$.
\end{lemma}

\begin{proof} 
	Note that the action of $\sigma$ has order $3$. Therefore we may assume that $B^3 = \mathbf{1}$. Suppose that $B$ is not diagonal. If $B$ has an eigenvalue $\lambda\neq 0$ of multiplicity $3$ and the corresponding eigenspace has dimension $1$, then 
	$B = \begin{pmatrix}
	\lambda & 1 & 0 \\
	0 & \lambda & 1 \\
	0 & 0 & \lambda
	\end{pmatrix} 
	$
	and $B^3 = \begin{pmatrix} \lambda^3 & 3\lambda^2 & 3\lambda \\ 0 & \lambda^3 & 3\lambda^2 \\0&0&\lambda^3\end{pmatrix}$, which is not diagonal. If $B$ has an eigenvalue $\lambda_1 \neq 0$ of multiplicity $2$ (or $3$) and the corresponding eigenspace has dimension $1$ (or $2$, respectively). Denote the remaining eigenvalue by $\lambda_2$. 
	Then $B = \begin{pmatrix} 
	\lambda_1 & 1 & 0 \\
	0 & \lambda_1 & 0 \\
	0 & 0 & \lambda_2\end{pmatrix}.$
	In this case $B^3 = \begin{pmatrix} \lambda_1^3 & 3 \lambda_1^2 & 0 \\
	0 & \lambda_1^3 & 0 \\
	0 & 0 & \lambda_2^3\end{pmatrix}. $ We conclude that $B$ is diagonal. The second part of the statement follows as $B^3 = \mathbf{1}$. 
\end{proof} 

In the following, using contradiction, we will show that $B$ has to be the identity matrix. Suppose first that $B$ has three distinct eigenvalues, i.e. $B = \begin{pmatrix} 1 & 0 & 0 \\
0 & \omega & 0 \\ 0 & 0 & \omega^2  \end{pmatrix} $ The eigenspaces of this action on $H^0(C_f, \mathcal{O}(3))$ are as follows: 
\begin{itemize} 
	\item The eigenspace of eigenvalue $1$ is spanned by $r^3, s^3, t^3, rst$. 
	\item The eigenspace of eigenvalue $\omega$ is spanned by $r^2 s, r t^2, s^2 t$. 
	\item The eigenspace of eigenvalue $\omega^2$ is spanned by $rt^2, r^2 t, st^2$.
\end{itemize} 
Comparing this with the previous list, we conclude that there is a linear relation $a rt^2 + b r^2t+ c st^2=0$ with $a,b,c \neq 0$, which is a contradiction.\\

Assume now that $B$ has two distinct eigenvalues, in other words $B$ may be assumed to be of the form $\begin{pmatrix} 1 & 0 & 0 \\ 0 & 1 & 0 \\ 0 & 0 & \omega^i\end{pmatrix}$ for $i =1$ or $i=2$. If $i=1$, then the embedding given by $\mathcal{L}$ is the same, as the one given by $\mathcal{O}(1)$, which is a contradiction as we assumed that $\mathcal{O}(1) \neq \mathcal{L}$. If $i=2$, then the eigenspaces of the action on $H^0(C_f, \mathcal{O}(3) )$ are as follows: 
\begin{itemize} 
	\item The eigenspace of eigenvalue $1$ is spanned by $r^3, r^2s, rs^2, s^3, t^3$.  
	\item The eigenspace of eigenvalue $\omega^{1}$ is spanned by 
	$rt^2, st^2$. 
	\item The eigenspace of eigenvalue $\omega^{2}$ is spanned by $r^2t, rst, s^2t$. 
\end{itemize} 
Comparing this with the previous list, we get a contradiction. We conclude that the assumption made at the beginning of this part was incorrect, i.e. $\phi^*\left( \mathcal{O}_{C_g} \right) = \mathcal{O}_{C_f}$ and hence by \Cref{thm:iso-of-Clifford-alg} we deduce the statement of \Cref{prop:iso-SBV}. \\

The above argument additionally implies the following proposition. 

\begin{prop}\label{prop:X-with-auto}
	Let $X$ be a genus one curve in $\mathbb{P}^2$. If $X$ is equipped with an order $3$ automorphism $\sigma$ over $k$, then $X$ is of the form $w^3-f(u,v)$ for some nondegnerate binary cubic form $f$. 
\end{prop} 

\begin{proof} 
	Let $E$ be the Jacobian of $X$. The automorphism $\phi$ induces an order three automorphism $\mu$ on $E$. In particular, as $E$ has a unique order three automorphism, we may assume that $\mu$ is given by $x \mapsto \omega x, y \mapsto y$ on the elliptic curve $E$. Furthermore, $\mu$ fixes three line bundles. Denote the non-trivial line bundle that is fixed by $\mathcal{L}$. \\
	Now $\phi$ defines an action on $H^0\left(X,\mathcal{O}(1)\right)$ and on $H^0\left(X, \mathcal{L}\right)$. As observed above, we may assume that $\phi$ acts on both by a matrix of the form 
	$\begin{pmatrix} 1&0 &0\\0&1&0\\0&0&\omega\end{pmatrix}.$
	Hence, there is a commutative diagram of the form 
	$$ \xymatrix{ X \ar[rd]\ar[r]^\phi & X \ar[d]\\
	&\mathbb{P}^1}.$$
Hence we may write $X$ as $w^3 - f(u,v)$ for a nondegenerate binary cubic form $f$. 
\end{proof}

\subsection{Action on the relative cubic curve $\mathcal{C}$ over $U$}

Recall that $\GL_2$ acts on $U$ by change of coordinates, for details see \Cref{sec:Action-on-U}. We first extend the action on $U$ to the curve $\mathcal{C}$ given by 
$$\mathcal{C}: w^3 - \left( x^3 u^3 + \alpha u^2v + \beta uv^2 + y^3 v^3 \right) \subseteq \mathbf{P}^2_U.$$ Consider the action of a matrix $\mathfrak{g} = \begin{pmatrix} a & b \\c& d \end{pmatrix}$ on $u,v$ given by a linear change of coordinates $\mathfrak{g}.u = au + bv, \mathfrak{g}.v = cu + dv$ and $\mathfrak{g}.w = w$. By a direct computation, we see that 
\begin{align*} 
	\mathfrak{g}.u^3 &= a^3 u^3 + 3a^2b u^2v + 3 ab^2uv^2 + b^3 v^3,\\
	\mathfrak{g}.u^2v &= a^2 c u^3 + a(ad+2bc) u^2v + b(bc+2ad) uv^2 + b^2d v^3 ,\\
	\mathfrak{g}.uv^2 &= ac^2 u^3 + c(bc+2ad) u^2v + d(ad+2bc) uv^2 + bd^2  v^3,\\
	\mathfrak{g}.v^3 &= c^3 u^3+ 3c^2d u^2v + 3cd^2 uv^2 + d^3 v^3.
\end{align*} 
Comparing this with the matrix in \cref{matrixaction}, we see that the following action extends the action of $\GL_2(k)$ on $U$ to $\mathbb{P}^2_U$. 
\begin{defn} 
	Let $\mathfrak{g} \in \GL_2(k)$ and $P=\left[ u:v:w\right] \in \mathbb{P}^2_U$. Define $$\mathfrak{g}*P = \left[ \mathfrak{g}^{-1}.u : \mathfrak{g}^{-1}.v: w \right].$$ 
\end{defn} 

\begin{lemma} 
	The action of $\GL_2(k)$ on $\mathbb{P}^2$ leaves the ideal $w^3 - x^3 u^3 - \alpha u^2 v - \beta uv^2 - y^3 v^3$ invariant. Therefore, there is a well-defined action of $\GL_2(k)$ on $\mathcal{C}$. 
\end{lemma} 

\begin{proof} 
	This can be proven using a direct computation. The action on $u^3, u^2v, uv^2,$ and $v^3$ cancels the action on $x^3, \alpha, \beta,$ and $y^3$, leaving the equation $w^3 - x^3 u^3 + \alpha u^2v + \beta uv^2 + y^2 v^3$ invariant. 
\end{proof} 

\subsection{Action on the relative elliptic curve $J$ over $U$} 

The above gives us an action of $\GL_2(k)$ on the elliptic curve $J$, that extends the action on $U$. Recall that $J$ is given by the equation $s^2 = \gamma^3 + \frac14 \Delta$. Direct computation shows that the action on $J$ is given by 
$$
\mathfrak{g}.\gamma =  \left( \det \mathfrak{g} \right)^2 \gamma, 
\qquad 
\mathfrak{g}.s =  \left( \det \mathfrak{g} \right)^3 s, \text{and} 
\qquad 
\mathfrak{g}.\Delta =  \left( \det \mathfrak{g} \right)^6 \Delta.  
$$

Suppose that $\mathfrak{g}$ fixes a diagonal form $f$ in $U$ given by $pu^3 + rv^3$. We computed the stabilizer of $f$ in \Cref{sec:IsoClifford}.  Let $\mathfrak{g} = \begin{pmatrix} 
	\omega^i & 0 \\
	0 & \omega^j \end{pmatrix}$ be an element in the stabilizer of $f$. If $i+j \equiv 0 (\mod 3)$, then $\mathfrak{g}$ acts on $J_f$ by the identity. If $i+j \not\equiv 0 (\mod 3)$, then $\mathfrak{g}$ fixes $s$ and takes $\gamma$ to $\omega^{2i+2j} \gamma$. This defines an automorphism of $J_f$ with three fixed points $0, (0,s_1), (0,-s_1)$, where $(0,s_1)$ is a three torsion point on $J_f$. 
Finally, suppose that $\mathfrak{g} = \begin{pmatrix} 0 & \lambda \\
	\frac1\lambda & 0 \end{pmatrix}$ is an element in the stabilizer of $f$. Recall that in this case $f$ has a zero over $k$. In this case, the automorphism induced on $J_f$ by $\mathfrak{g}$ is the isogeny $[-1]$ that takes a point $(s,\gamma)$ to $(-s,\gamma)$.\\

In conclusion, every $f \in U$ has finite stabilizer and every element $\mathfrak{g}$ in the stabilizer of $f$ has finite fixed points in $J_f$. This implies the following proposition. 
\begin{prop}
	The stack $\left[J/\GL_2\right]$ is a Deligne-Mumford stack and each orbit in $\left[J/\GL_2\right](k)$ defines a relative elliptic curve over an orbit in $\left[ U/\GL_2\right](k)$. 
\end{prop} 

\subsection{Proof of \Cref{thm:Moduli-space}} 

We are now ready to proof \Cref{thm:Moduli-space}. Recall the construction of our bijection from the beginning of this section. We need to show that it is well-defined. 

\begin{lemma}\label{lemma:kappa-alpha-f-is-C-f}
	Let $f$ be a binary cubic nondegenerate form. Denote by $C_f$ the curve defined by $w^3- f(u,v)$, $E_f$ the Jacobian of $C_f$, and $\alpha_f$ the Brauer class associated to the Clifford algebra of $f$. Then the map $\kappa$ of \cref{eq:SES} takes $\alpha_f$ to the class of the principal homogeneous space $C_f$. 
\end{lemma} 

\begin{proof} 
	This was shown in \cite[p. 518]{Haile:WhenistheCliffordAlgebraofabinarycubicformsplit}
	even though it was stated in different terms. Furthermore, it is a a special case of the computations done in \cite[Proposition 3.2]{Kulkarni:the-extension-of-the-reduced-clifford-algebra-and-its-brauer-class}. We will review the basic argument here for completion. 
	Fix a point $P \in C_f(\kbar)$. Then a cocycle representation corresponding to $C_f$ is given by $f(\sigma) = P^{\sigma^{-1}} - P$. We will show that this is the image of $\alpha_f$. Let us first review the definition of $\kappa$ from sequence \cref{eq:SES}. This construction can be found in more detail for example in \cite{Lichtenbaum}. Denote $\overline{E} = E \times_k \kbar$. By Tsen's theorem, the group $\Br(E) $ is isomorphic to $H^2\left( k, k(\overline{E})^\times \right)$ and therefore we may think of $\alpha$ as a cocycle in this cohomology group. Consider the exact sequence 
	$$\xymatrix{ 0 \ar[r]& k(\overline{E})^\times \ar[r]& \Prin(\Ebar) \ar[r] & \Div(\Ebar) \ar[r] &0  },$$ 
	where $\Div(\overline{E})$ is the group of divisors and $\Prin(\overline{E})$ the set of principal divisors on $\overline{E}$. This induces a map 
	$$\xymatrix{ H^2 \left(k, k(\overline{E})^\times \right)  \ar[r] & H^2\left( k, \Prin(\Ebar) \right)}.$$
	Denote the image of $\alpha_f$ under this map by $h_f$. Now consider the sequence 
	$$\xymatrix{0 \ar[r] & \Prin(\Ebar) \ar[r] & \Div^0(\Ebar) \ar[r] & E(\kbar) \ar[r] & 0 }.$$
	This induces a map 
	\begin{equation}\label{eq:H1toH2}\xymatrix{H^1\left( k, E(\kbar)\right) \ar[r] & 
		H^2\left( k, \Prin(\Ebar) \right)}.\end{equation} 
	It can be shown that there is a unique lift of $\alpha_f$, which is $\kappa(\alpha_f)$. 
	Recall that $\alpha_f$ is defined via the \Poincare bundle and the $\Theta$-divisor $\Theta_0$. In $H^2\left( k, \kbar(E)^\times\right)$, the Azumaya algebra $\alpha_f$ can be represented by the cocycle that takes $\sigma, \tau \in \Gal(k)$ to a function whose divisor is 
	\begin{align} \label{eqn:theta}\Theta_{[P - P^{\sigma^{-1}} ]} + \Theta_{[P^{\sigma^{-1}} - P^{\sigma^{-1}\tau^{-1}} ]} - \Theta_{[P - P^{\sigma^{-1}\tau^{-1}0} ]} - \Theta. \end{align}
	Here $\Theta_a= t_a \Theta = \Theta +a$ denotes translation of the theta divisor by an element $a$ in the Jacobian $\overline{E}$. By definition, equation \ref{eqn:theta} describes the image of the cocycle of $C_f$ under the map in sequence \ref{eq:H1toH2}. 
\end{proof}

We now proceed with our investigation of the image under $\kappa$ of a Brauer class that is invariant under complex multiplication. Let $E$ be an elliptic curve over $k$ with $j$-invariant $0$. The automorphism group of $E$ is $\ZZ/6\ZZ$ by \cite[III Theorem 10.1]{silverman}. Denote the order $3$ automorphism  of $E$ by $\theta$. Note that if $E$ is given by the affine equation $y^2= x^3 + \frac{1}{4} \Delta$, then $\theta$ is the morphism defined by $\theta(x,y) = (\omega x, y)$. The fixed points of $\theta$ form the subgroup $\mathcal{T} = \left\{ 0, \left(0, \tfrac{1}{2} \sqrt{\Delta}\right), \left(0, -\tfrac{1}{2} \sqrt{\Delta}\right)\right\}$ of the $3$-torsion of $E$. Finally, denote by $\lambda = \theta - [1]$ the isogeny with kernel $\mathcal{T}$. Consider the exact sequence of $G= \Gal(\kbar/k)$-modules 
$$\xymatrix{ 0 \ar[r] & \mathcal{T} \ar[r] & E \ar[r]^\lambda \ar[r] & E \ar[r] & 0 }.$$
Applying group cohomology yields the exact sequence 
$$\xymatrix{0 \ar[r]& E(k)/\lambda E(k) \ar[r] & H^1\left( k, \mathcal{T}\right) \ar[r]^{\delta\phantom{ksjl}} & H^1\left( k, E(\kbar)\right)[\lambda] \ar[r] & 0,}$$
where $H^1\left( k, E(\kbar)\right)[\lambda]$ denotes the classes $[X]$ in the kernel of $\lambda_*:  H^1\left( k, E(\kbar)\right) \rightarrow H^1\left( k, \mathcal{T}\right)$. By \Cref{prop:torsor-covering}, every such $X$ is in the image of the map $\delta$. By \cite[Theorem 6.1]{Haile-Wadsworth-Han:Cyclic-Twists}, the image of this map $\delta$ is exactly described by principal homogeneous spaces of the form $w^3-f(u,v)$. As a $3$-torsion element in $H^1\left(k, E(\kbar)\right)$ is invariant under $\theta$ if and only if  it is in trivial under $\lambda_*$, the statement follows from a computational result of Haile, Han, and Wadsworth \cite[Theorem 6.1]{Haile-Wadsworth-Han:Cyclic-Twists}. 
We will prove the statement geometrically for completion.

\begin{lemma}\label{lemma:kappa-alpha-is-binary-cubic}
	Every $3$-torsion element in $H^1(k, E(\kbar))$ that is invariant under $\theta$ can be described as the principal homogeneous space given by a curve $C_f : w^3-f(u,v)$ for some nondegenerate binary cubic form $f$ so that the Jacobian of $C_f$ is isomorphic to $E$. 
\end{lemma} 

\begin{proof} 
	Let $X$ be a $3$-torsion element in $H^1(k, E(\kbar))$  that is invariant under $\theta$. The canoncial map $X \rightarrow \theta_*(X) \cong X$ gives an order $3$ automorphism $\sigma$ of $X$. As seen in \Cref{prop:X-with-auto}, this implies that $X$ is of the form $w^3-f(u,v)$ for some nondegenerate binary cubic form $f$. 	
\end{proof} 

%
%

\begin{proof}[{Proof of \Cref{thm:Moduli-space}}]
We recall the bijection from the beginning of this section. Let $f$ be a nondegerate binary cubic form. Let $C_f$ be the cubic planar curve $w^3-f(u,v)$ and let $E_f$ be the Jacobian of $C_f$. Furthermore, let $\alpha_f \in \prescript{}{3}{\Br(E_f)}$ be the Brauer class on $E_f$ associated to the Clifford algebra. For another binary cubic form $g$ in the same orbit as $f$ with respect to the action of $\GL_2(k)$, define $C_g, E_g,$ and $\alpha_g$ respectively. Then there is some $\mathfrak{g} \in \GL_2(k)$ so that $\mathfrak{g}.f = g$, where we recall that $\mathfrak{g}$ acts on the coefficients of $f$. By \cite{Haile:On-the-Clifford-Algebra-of-a-binary-cubic-form} the elliptic curve $E_f$ is defined by the equation 
	$$s_f^2 = \gamma_f^3 + \frac14\Delta_f$$ so that by our computations in the previous section $E_g$ is given by $$\det(\mathfrak{g})^6 s_f^2 = \det(\mathfrak{g})^6 \gamma_f^3 + \frac14 \det(\mathfrak{g})^6 \Delta_f.$$ This implies by \cite[Section III.1]{silverman} that $E_f \cong E_g$ over $k$. Furthermore, for any  $\mathfrak{g} = \begin{pmatrix} a& b \\ c & d \end{pmatrix}$, the action $\mathfrak{g}.x = ax + cy$ and $\mathfrak{g}.y = bx + dy$ defines an isomorphism from the Clifford algebra of $f$ to the Clifford algebra of $g$. On the center this action takes the coordinate ring of $E_f$ to the coordinate ring of $E_g$. Hence, the equivalence class of $(E_f, \alpha_f)$ is well-defined for an orbit under the action of $\GL_2(k)$. \\
	
	On the other hand, consider a pair $(E,\alpha)$ as in the statement of the theorem. Consider the image of $\alpha$ under $\kappa$. By \Cref{lemma:kappa-alpha-is-binary-cubic} there is some binary cubic form $f$ so that $\kappa(\alpha)$ is represented by the principal homogeneous space given by $C_f$. Note that we get a well-defined class modulo the action of $\GL_2(k)$. 	Suppose that $E$ is isomorphic to $E'$ via the isomorphism $\phi: E' \rightarrow E$. Suppose that $g$ is constructed in a similar way from $E'$. By naturality of the spectral sequence, the following diagram commutes
	$$\xymatrix{
		0 \ar[r] &  \Br(k) \ar@{=}[d] \ar[r] & \Br(E) \ar[d]^{\phi^*} \ar[r]& H^1\left( k, E(\kbar)\right) \ar[d]^{\phi^*} \ar[r] &0 \\
		0 \ar[r] &  \Br(k) \ar[r] & \Br(E') \ar[r]& H^1\left( k, E'(\kbar)\right)\ar[r] &0 }.
	$$
	In conclusion, the equivalence class of $(E,\alpha)$ gives us a well-defined orbit in $U$ modulo $\GL_2(k)$. \\
	
	Finally, start with an orbit of $f$ and let $(E_f,\alpha_f)$ be as above. By \cref{lemma:kappa-alpha-f-is-C-f}, the orbit of binary cubic forms associated to the pair $(E_f,\alpha_f)$ is again $f$. On the other hand, start with a pair $(E, \alpha)$ and an associated orbit $f$ as in the second part of the proof. Then the Jacobian of $C_f$ is $E$ by definition. Furthemore, $\alpha$ is a lift of $C_f$ to $\Br E$ and so $\alpha$ and $\alpha_f$ describe the same element in $\Br(E)/\Br_0(E)$ by the short exact sequence in \cref{eq:SES}. 
\end{proof} 

\subsection{Descending the Brauer Class}

Let $U$, $\mathcal{C}$, and $J$ be as before. A Brauer class over $J$ descends to a Brauer class of the quotient stack $\left[J/\GL_2\right]$ if and only if it is $\GL_2$-equivariant \cite[Theorem 4.46]{FGA-Vistoli}. We finish this section by proving that the Brauer class $\alpha \in \Br(J)$ associated to the binary cubic generic Clifford algebra descends. 

\begin{prop}\label{prop:descend-alpha} 
	The Brauer class of $\alpha$ descends to the quotient stack $\left[J/\GL_2\right] \rightarrow \left[ U/\GL_2\right]$. 
\end{prop} 

\begin{proof} 
	Let $f$ and $g$ be nondegenerate binary cubic forms and $\mathfrak{g} \in \GL_2(k)$ so that $f = \mathfrak{g}.g$. Then the Jacobians of $C_f$ and $C_g$ are isomorphic over $k$ and this isomorphism induces an identification of $\alpha_f$ and $\alpha_g$ in $\Br J_f$ and $\Br J_g$, respectively. Hence the Brauer class $\alpha$ is $\GL_2$-equivariant and descends to the quotient. 
\end{proof} 

\appendix

\section{Coverings of Elliptic Curves}\label{appendix}

Let $k$ be a field of characteristic different from $2$ or $3$ and let $E$ be an elliptic curve $E$. Fix an automorphism $\lambda: E \rightarrow E$ of finite order $n$. Denote the kernel of $\lambda$ by $E[\lambda]$. Consider the exact sequence 
$$\xymatrix{ 0 \ar[r] & E[\lambda] \ar[r] & E(\kbar) \ar[r] & E(\kbar) \ar[r] & 0}.$$
This induces an exact sequence on Galois-cohomology 
$$\xymatrix{ 0 \ar[r]& E(k)/\lambda E(k) \ar[r] &H^1(k, E[\lambda]) \ar[r]^\delta & H^1(k, E(\kbar)) \ar[r] & H^1(k, E(\kbar))\ar[r] & \cdots}.$$
In our proof of the correspondence of moduli spaces, the image of the map $\delta$ will play a crucial role. We will describe elements in the image using a generalization of $n$-coverings which were introduced by Cassels in \cite{Cassels}. Such a notion can be made in more generality for abelian varieties.

\begin{defn} 
	A \emph{$\lambda$-covering} is a pair $(X,\psi)$, where $X$ is a $k$-torsor under $E$, and $\psi: X \rightarrow E$ is a morphism satisfying 
	$$\psi(P+x) = \lambda(P) + \psi(x)$$
	for all $P \in E(\kbar)$ and $x \in X(\kbar)$. That is, the following diagram commutes
	$$\xymatrix{ 
		E \times X \ar[d]^{+} \ar[r]^{(\lambda, \psi)} & E \times E \ar[d]^+ \\
		X \ar[r]^\psi & E
	}.$$
	Let $(X,\psi)$ and $(Y,\phi)$ be $\lambda$-coverings. A morphism $(X,\psi) \rightarrow (Y,\phi)$ is given by a morphism $\eta: X \rightarrow Y$ such that the following diagram commutes: 
	$$
	\xymatrix{
		& E\times Y\ar[rd]^{(\lambda,\phi)}] \ar[dd]|(.35)\hole|(0.4)\hole|(0.45)\hole|(0.5)\hole|(0.55)\hole^(0.7){+} \\
		E \times X \ar[ru]^{(\id,\eta)} \ar[dd]^+ \ar[rr]^{(\lambda,\psi)} && E\times E \ar[dd]^{+}	\\
		& Y \ar[rd]^\phi \\
		X \ar[rr]^\psi \ar[ru]^\eta && E	
	}.
	$$
\end{defn} 

\begin{prop}\label{prop:torsor-covering}
	A $k$-torsor $X$ can be endowed with the structure of a $\lambda$-covering if and only if its class in $H^1(k,E(\overline{k}))$ is trivial under the push-forward $\lambda_*$. In particular, there is a one-to-one correspondence between isomorphism classes of $\lambda$-coverings and the image of $\delta$. \end{prop} 

\begin{proof}The proof is analogous to the proof of \cite[Proposition 3.3.4]{Skorobogatov}. 
	We review it here for completion. Let $(X, \psi)$ be a $\lambda$-covering of $E$ and consider the push-forward $\lambda_*X$. This push-forward is defined to be  
	$$\lambda_*(X) = \left( X\times_k E \right)/E,$$
where the action of $E$ is given by 
	$$P.(x,Q) = \left( - (P+x), \lambda(P) + Q\right),$$
	for $P,Q \in E(\overline{k})$ and $x \in X(\overline{k})$. Notice that 
	$$P.\left(x,-\psi(x)\right) = \left( -(P + x), \lambda(P) - \psi(x) \right) = \left( - (P+x), \psi\left( P+x\right)\right).$$
	Thus the set $\left( \text{Id}, - \psi\right) \subseteq X \times_k E$ defines a $k$-point on the quotient so that the torsor $\lambda_*(X)$ is trivial. On the other hand, when $\lambda_*(X) =0$, then by definition $\lambda_*(X) \cong E$. The canonical map $\psi: X \rightarrow \lambda_*(X)$ defines a $\lambda$-covering. 
\end{proof}

\bibliography{newbib}

\begin{thebibliography}{10}

\bibitem{Auslander-Goldman:Maximal-Orders}
Maurice Auslander and Oscar Goldman.
\newblock Maximal orders.
\newblock {\em Trans. Amer. Math. Soc.}, 97:1--24, 1960.

\bibitem{BhargavaGross}
Manjul Bhargava and Benedict~H. Gross.
\newblock Arithmetic invariant theory.
\newblock In {\em Symmetry: representation theory and its applications}, volume
  257 of {\em Progr. Math.}, pages 33--54. Birkh\"{a}user/Springer, New York,
  2014.

\bibitem{Bosch:Neron-Models}
Siegfried Bosch, Werner L\"utkebohmert, and Michel Raynaud.
\newblock {\em N\'eron models}, volume~21 of {\em Ergebnisse der Mathematik und
  ihrer Grenzgebiete (3) [Results in Mathematics and Related Areas (3)]}.
\newblock Springer-Verlag, Berlin, 1990.

\bibitem{Cassels}
J.~W.~S. Cassels.
\newblock Arithmetic on curves of genus {$1$}. {III}. {T}he
  {T}ate-\v{S}afarevi\v{c} and {S}elmer groups.
\newblock {\em Proc. London Math. Soc. (3)}, 12:259--296, 1962.

\bibitem{Chan-Young-Zhang:Discriminant-formulas-and-applications}
Kenneth Chan, Alexander~A. Young, and James~J. Zhang.
\newblock Discriminant formulas and applications.
\newblock {\em Algebra Number Theory}, 10(3):557--596, 2016.

\bibitem{Childs:Linearizing-of-n-ic-forms-and-generalized-clifford-algebras}
Lindsay~N. Childs.
\newblock Linearizing of {$n$}-ic forms and generalized {C}lifford algebras.
\newblock {\em Linear and Multilinear Algebra}, 5(4):267--278, 1977/78.

\bibitem{Ciperiani-Krashen:Relative-Brauer-Groups-of-geneus-1-curves}
Mirela Ciperiani and Daniel Krashen.
\newblock Relative {B}rauer groups of genus 1 curves.
\newblock {\em Israel J. Math.}, 192(2):921--949, 2012.

\bibitem{Davenport1}
H.~Davenport.
\newblock On the class-number of binary cubic forms. {I}.
\newblock {\em J. London Math. Soc.}, 26:183--192, 1951.

\bibitem{Davenport2}
H.~Davenport.
\newblock On the class-number of binary cubic forms. {II}.
\newblock {\em J. London Math. Soc.}, 26:192--198, 1951.

\bibitem{Eisenstein}
G.~{Eisenstein}.
\newblock {Untersuchungen \"uber die cubischen Formen mit zwei Variabeln}.
\newblock {\em {J. Reine Angew. Math.}}, 27:89--104, 1844.

\bibitem{Gan-Gross-Savin:Fourier-Coefficients-of-modular-forms}
Wee~Teck Gan, Benedict Gross, and Gordan Savin.
\newblock Fourier coefficients of modular forms on {$G_2$}.
\newblock {\em Duke Math. J.}, 115(1):105--169, 2002.

\bibitem{Gelfand-Kapranov-Zelevinski:Discriminats-resultants-and-multidimensional-determinants}
I.~M. Gelfand, M.~M. Kapranov, and A.~V. Zelevinsky.
\newblock {\em Discriminants, resultants, and multidimensional determinants}.
\newblock Mathematics: Theory \& Applications. Birkh\"{a}user Boston, Inc.,
  Boston, MA, 1994.

\bibitem{Grothendieck:Le-groupe-de-Brauer-II}
Alexander Grothendieck.
\newblock Le groupe de {B}rauer. {II}. {T}h\'{e}orie cohomologique [
  {MR}0244270 (39 \#5586b)].
\newblock In {\em S\'{e}minaire {B}ourbaki, {V}ol. 9}, pages Exp. No. 297,
  287--307. Soc. Math. France, Paris, 1995.

\bibitem{Grothendieck:Technique-de-descente-et-theoremes-dexistence-en-geometrie-II}
Alexander Grothendieck.
\newblock Technique de descente et th\'eor\`emes d'existence en g\'eom\'etrie
  alg\'ebrique. {II}. {L}e th\'eor\`eme d'existence en th\'eorie formelle des
  modules.
\newblock In {\em S\'eminaire {B}ourbaki, {V}ol.\ 5}, pages Exp.\ No.\ 195,
  369--390. Soc. Math. France, Paris, 1995.

\bibitem{Haile:On-the-Clifford-Algebra-of-a-binary-cubic-form}
Darrell~E. Haile.
\newblock On the {C}lifford algebra of a binary cubic form.
\newblock {\em Amer. J. Math.}, 106(6):1269--1280, 1984.

\bibitem{Haile:WhenistheCliffordAlgebraofabinarycubicformsplit}
Darrell~E. Haile.
\newblock When is the {C}lifford algebra of a binary cubic form split?
\newblock {\em J. Algebra}, 146(2):514--520, 1992.

\bibitem{Haile-Wadsworth-Han:Cyclic-Twists}
Darrell~E. Haile, Ilseop Han, and Adrian~R. Wadsworth.
\newblock Curves {$\mathcal{C}$} that are cyclic twists of {$Y^2=X^3+c$} and
  the relative {B}rauer groups {$Br(k(\mathcal{C})/k)$}.
\newblock {\em Trans. Amer. Math. Soc.}, 364(9):4875--4908, 2012.

\bibitem{Hartshorne}
Robin Hartshorne.
\newblock {\em Algebraic geometry}.
\newblock Springer-Verlag, New York-Heidelberg, 1977.
\newblock Graduate Texts in Mathematics, No. 52.

\bibitem{Heerema:an-algebra-determined-by-a-binary-cubic-form}
Nickolas Heerema.
\newblock An algebra determined by a binary cubic form.
\newblock {\em Duke Math. J.}, 21:423--443, 1954.

\bibitem{Katz-Mazur}
Nicholas~M. Katz and Barry Mazur.
\newblock {\em Arithmetic moduli of elliptic curves}, volume 108 of {\em Annals
  of Mathematics Studies}.
\newblock Princeton University Press, Princeton, NJ, 1985.

\bibitem{Kulkarni:the-extension-of-the-reduced-clifford-algebra-and-its-brauer-class}
Rajesh~S. Kulkarni.
\newblock The extension of the reduced {C}lifford algebra and its {B}rauer
  class.
\newblock {\em Manuscripta Math.}, 112(3):297--311, 2003.

\bibitem{Kulkarni:On-the-Clifford-algebra-of-a-binary-form}
Rajesh~S. Kulkarni.
\newblock On the {C}lifford algebra of a binary form.
\newblock {\em Trans. Amer. Math. Soc.}, 355(8):3181--3208, 2003.

\bibitem{Lichtenbaum}
Stephen Lichtenbaum.
\newblock Duality theorems for curves over {$p$}-adic fields.
\newblock {\em Invent. Math.}, 7:120--136, 1969.

\bibitem{Milne:Jacobian-Varieties}
J.~S. Milne.
\newblock Jacobian varieties.
\newblock In {\em Arithmetic geometry ({S}torrs, {C}onn., 1984)}, pages
  167--212. Springer, New York, 1986.

\bibitem{Milne:Etale-Cohomology}
James~S. Milne.
\newblock {\em \'{E}tale cohomology}, volume~33 of {\em Princeton Mathematical
  Series}.
\newblock Princeton University Press, Princeton, N.J., 1980.

\bibitem{Mumford:Abelian-Varieties}
David Mumford.
\newblock {\em Abelian varieties}.
\newblock Tata Institute of Fundamental Research Studies in Mathematics, No. 5.
  Published for the Tata Institute of Fundamental Research, Bombay; Oxford
  University Press, London, 1970.

\bibitem{Revoy:AlgebresdeCliffordetalgebresexterieures}
Ph. Revoy.
\newblock Alg\`ebres de {C}lifford et alg\`ebres ext\'{e}rieures.
\newblock {\em J. Algebra}, 46(1):268--277, 1977.

\bibitem{Roby:AlgebresdeClifforddesformespolynomes}
Norbert Roby.
\newblock Alg\`ebres de {C}lifford des formes polynomes.
\newblock {\em C. R. Acad. Sci. Paris S\'{e}r. A-B}, 268:A484--A486, 1969.

\bibitem{schur}
Issai Schur.
\newblock {\em Vorlesungen \"uber Invariantentheorie}.
\newblock Grundlehren der mathematischen Wissenschaften. Springer, 1968.

\bibitem{silverman}
Joseph~H. Silverman.
\newblock {\em The arithmetic of elliptic curves}, volume 106 of {\em Graduate
  Texts in Mathematics}.
\newblock Springer, Dordrecht, second edition, 2009.

\bibitem{Skorobogatov}
Alexei Skorobogatov.
\newblock {\em Torsors and rational points}, volume 144 of {\em Cambridge
  Tracts in Mathematics}.
\newblock Cambridge University Press, Cambridge, 2001.

\bibitem{stacks-project}
The {Stacks Project Authors}.
\newblock {\itshape Stacks Project}.
\newblock \url{http://stacks.math.columbia.edu}, 2018.

\bibitem{Sylvester}
J.~J. Sylvester and F.~Franklin.
\newblock Tables of the {G}enerating {F}unctions and {G}roundforms for
  {S}imultaneous {B}inary {Q}uantics of the {F}irst {F}our {O}rders, {T}aken
  {T}wo and {T}wo {T}ogether.
\newblock {\em Amer. J. Math.}, 2(4):293--306, 1879.

\bibitem{FGA-Vistoli}
Angelo Vistoli.
\newblock Grothendieck topologies, fibered categories and descent theory.
\newblock In {\em Fundamental algebraic geometry}, volume 123 of {\em Math.
  Surveys Monogr.}, pages 1--104. Amer. Math. Soc., Providence, RI, 2005.

\bibitem{Wang-Wang:A-note-on-generic-Clifford-algebras-of-binary-cubic-forms}
Linhong Wang and Xingting Wang.
\newblock A note on generic {C}lifford algebras of binary cubic forms.
\newblock {\em Algebr. Represent. Theory}, 23(4):1797--1806, 2020.

\bibitem{Wang-Wu:AclassofASregularalgebrasofdimensionfive}
S.-Q. Wang and Q.-S. Wu.
\newblock A class of {AS}-regular algebras of dimension five.
\newblock {\em J. Algebra}, 362:117--144, 2012.

\bibitem{Williamson:Crossed-products-and-hereditary-orders}
Susan Williamson.
\newblock Crossed products and hereditary orders.
\newblock {\em Nagoya Math. J.}, 23:103--120, 1963.

\bibitem{Wood:Rings-and-ideals}
Melanie~Matchett Wood.
\newblock Rings and ideals parameterized by binary {$n$}-ic forms.
\newblock {\em J. Lond. Math. Soc. (2)}, 83(1):208--231, 2011.

\end{thebibliography}
\bibliographystyle{plain}

\end{document}